\documentclass[11pt]{article}

\usepackage{amssymb,amsmath,euscript,bbm,xcolor,graphicx,epstopdf}
\usepackage{verbatim} 
\newtheorem{proposition}{Proposition}[section]
\newtheorem{lemma}[proposition]{Lemma}
\newtheorem{theorem}[proposition]{Theorem}

\newtheorem{corollary}[proposition]{Corollary}

\def\la{\lambda}
\def\La{\Lambda}
\def\ep{\varepsilon}

\def\l{{\langle}}
\def\r{\rangle}

\def\R{{\mathbb R}}
\def\S{{\mathbb S}}

\def\E{{\mathbb E}}
\def\P{{\mathbb P}}

\newcommand{\GOI}{{\rm GOI}}
\newcommand{\GOE}{{\rm GOE}}

\makeatletter \@addtoreset{equation}{section} \makeatother
\newcommand {\qed}%
{%
    {}\hfill
    {}\hfill
    {$\square $}%
    \vspace {0.3cm}%
    \pagebreak [2]%
    \par
}%
\newenvironment{proof}[1]{%
    \vspace{0.3cm}%
    \pagebreak [2]%
    \par%
    \noindent {\bf  Proof~#1\ }}{\qed}%
\newenvironment{example}{%
    \vspace{0.3cm} \pagebreak [2]%
    \par%
    \refstepcounter{proposition}%
    \noindent%
    {\bf  Example~\theproposition\ }}{}%
\newenvironment{remark}{%
    \vspace{0.3cm} \pagebreak [2]%
    \par%
    \refstepcounter{proposition}
    \noindent%
    {\bf Remark~\theproposition\  }}{}%
\begin{document}

\title {Expected Number and Height Distribution of Critical Points of Smooth Isotropic Gaussian Random Fields \thanks{Research partially
supported by NIH grant R01-CA157528.}}
\author{Dan Cheng\\ Texas Tech University
 \and Armin Schwartzman \\ University of California San Diego}

\maketitle

\begin{abstract}
We obtain formulae for the expected number and height distribution of critical points of smooth isotropic Gaussian random fields parameterized on Euclidean space or spheres of arbitrary dimension. The results hold in general in the sense that there are no restrictions on the covariance function of the field except for smoothness and isotropy. The results are based on a characterization of the distribution of the Hessian of the Gaussian field by means of the family of Gaussian orthogonally invariant (GOI) matrices, of which the Gaussian orthogonal ensemble (GOE) is a special case. The obtained formulae depend on the covariance function only through a single parameter (Euclidean space) or two parameters (sphere), and include the special boundary case of random Laplacian eigenfunctions.
\end{abstract}

\noindent{\small{\bf Keywords}: Gaussian random fields; Isotropic; Critical points; Height density; Random matrices; GOI; GOE; Boundary; Kac-Rice formula; Sphere.}

\noindent{\small{\bf Mathematics Subject Classification}:\ 60G15, 60G60, 15B52.}

\section{Introduction}
Computing the expected number of critical points of smooth Gaussian random fields is an important problem in probability theory \cite{Adler81,AT07,CL67} and has extensive applications in various areas such as physics \cite{Annibale:2003,Cavagna:1997,Crisanti:2004,Halperin:1966}, statistics \cite{ATW12,CL67,Lindgren72,CS16,CS16sphere}, neuroimaging \cite{Nichols:2003,Taylor:2007,Worsley:1996b,Worsley:2004}, oceanography \cite{L-H60,Lindgren82} and astronomy \cite{AstrophysJ04,AstrophysJ85}. Many researchers from different areas have worked on this problem and created certain powerful tools, including the famous Kac-Rice formula \cite{Rice:1945,AT07}. Although one can use the Kac-Rice formula to find an implicit formula for the expected number of critical points, it remains difficult to evaluate the expectation explicitly for most smooth Gaussian random fields defined on Euclidean space $\R^N$ or the $N$-dimensional unit sphere $\S^N$ when $N>1$.

An exciting breakthrough was made by Fyodorov \cite{Fyodorov04}, making the explicit evaluation available for a large class of isotropic Gaussian random fields. The main novel idea was to write the Hessian of the Gaussian field as a Gaussian random matrix involving the Gaussian Orthogonal Ensemble (GOE). This has lead to many important applications and further developments \cite{Bray07,Fyodorov12,Fyodorov07,Fyodorov15}, including the study of critical points of spin glasses \cite{AuffingerPhD, Auffinger:2013a,Auffinger:2013b,Fyodorov14}, which is related to Gaussian random fields on $\S^N$.

However, the result of Fyodorov \cite{Fyodorov04} and its existing further developments do not apply to all isotropic fields but are in fact restricted to the class of fields for which the algebraic form of the covariance function of the Gaussian field needs to give a valid covariance function in $\R^N$ or $\S^N$ for every $N\ge 1$ \cite{AzaisW08,CS15}. This restriction appears naturally in physics when one is interested in studying the asymptotics as $N\to \infty$, but it becomes very limiting in recent applications in statistics \cite{CS16,CS16sphere} and astronomy \cite{CMW}, where the interested objects are Gaussian random fields defined on $\R^{N_0}$ or $\S^{N_0}$ for some specific $N_0$. Fyodorov's restriction ignores those covariance functions whose forms are valid in $\R^{N_0}$ ($\S^{N_0}$ respectively) but invalid in $\R^N$ ($\S^N$ respectively) when $N>N_0$, which correspond to a large class of Gaussian random fields.
Motivated by these real applications and the completeness of the theory as well, we show in this paper that the restriction can be in fact removed, so that the explicit evaluation of the expected number of critical points is available for general isotropic Gaussian random fields.

More specifically, let $X=\{X(t), t\in T\}$ be a centered, unit-variance, smooth isotropic Gaussian random field, where $T$ is $\R^N$ or $\S^N$. Let
\begin{equation}\label{Eq:mu-i}
\mu_i(X, u) = \# \left\{ t\in D: X(t)\geq u, \nabla X(t)=0, \text{index} (\nabla^2 X(t))=i \right\}, \quad 0\le i\le N,
\end{equation}
where $D$ is an $N$-dimensional unit-area disc on $T$, $\nabla X(t)$ and $\nabla^2 X(t)$ are respectively the gradient and Hessian of $X$, and $\text{index} (\nabla^2 X(t))$ denotes the number of negative eigenvalues of $\nabla^2 X(t)$. That is, $\mu_i(X, u)$ is the number of critical points of index $i$ of $X$ exceeding $u$ over the unit-area disc $D$. Our first goal is to compute $\E[\mu_i(X, u)]$ for general smooth isotropic Gaussian fields, especially to remove the restriction in \cite{Fyodorov04}. The main tool is still the Kac-Rice formula. However, by investigating the covariance structure of general isotropic Gaussian fields, we find that the problem can be solved by writing the Hessian as a class of Gaussian random matrices called Gaussian Orthogonally Invariant (GOI) matrices \cite{SMT08} or isotropic matrices \cite{Chevillard}. This is the class of matrices $M$ whose distribution is invariant under all transformations of the form $QMQ^T$ with orthogonal $Q$. This class of Gaussian random matrices, originally introduced by Mallows \cite{Mallows61}, extends GOE matrices in the sense that GOE is a special case when the diagonal entries are independent. The additional dependence in GOI is captured by a covariance parameter $c \ge -1/N$.

Using this construction (see Lemmas \ref{Lem:GOE for det Hessian} and \ref{Lem:GOE for det Hessian sphere}), we can write the Kac-Rice integral in terms of the density of the ordered GOI eigenvalues and obtain general implicit computable formulae of $\E[\mu_i(X, u)]$ (see Theorems \ref{Thm:critical-points-nondegenerate-euclidean}, \ref{Thm:critical-points-degenerate-euclidean}, \ref{Thm:critical-points-nondegenerate-sphere} and \ref{Thm:critical-points-degenerate-sphere}). Because of the isotropy assumption, the obtained formulae depend on the covariance function only through its first and second derivatives at zero. Explicit calculations are shown for isotropic Gaussian fields on $\R^2$ and $\S^2$.

As we shall see here, Fyodorov's construction corresponds to the subset of GOI matrices corresponding to $c \ge 0$, see \eqref{Eq:c-positive}. The resulting restriction can be characterized by the covariances of the first and second derivatives of the field and it reduces to simple constraints on a single parameter $\kappa$ in the case of $\R^N$ ($0 < \kappa^2 \le 1$) and two parameters $\eta$ and $\kappa$ in the case of $\S^N$ ($\kappa^2 - \eta^2 \le 1$) (see Sections \ref{sec:restriction-Euclidean} and \ref{sec:restriction-sphere}). By removing this restriction, we are able to obtain results for the entire range of these parameters ($0 < \kappa^2 \le (N + 2)/N$ and $\kappa^2 - \eta^2 \le (N + 2)/N$). 

Of special interest is the case of random eigenfunctions of the Laplace operator, obtained at the boundary of the parameter space ($\kappa^2 = (N + 2)/N$ and $\kappa^2 - \eta^2 = (N + 2)/N$). These random fields satisfy the Helmholtz partial differential equation and have a degenerate covariance between the field and its Hessian. On the sphere, these become random spherical harmonics, which have been widely studied \cite{CM,CMW,Wigman:2009} due to applications in physics and astronomy. We obtain results for this boundary case as well using a different technique involving the Helmholtz equation.

The second goal of this paper is to obtain the height distribution of critical points. Define the height distribution of a critical value of index $i$ of $X$ at some point, say $t_0$, as
\begin{equation}\label{Eq:F}
F_i(u) :=\lim_{\ep\to 0} \P\left\{X(t_0)>u | \exists \text{ a critical point of index $i$ of } X(t) \text{ in } B(t_0, \ep) \right\},
\end{equation}
where $B(t_0, \ep)$ is the geodesic ball on $T$ of radius $\ep$ centered at $t_0$. Such distribution has been of interest for describing fluctuations of the cosmic background in astronomy \cite{AstrophysJ85,AstrophysJ04} and describing the height of sea waves in oceanography \cite{L-H52,L-H80,Lindgren82,Sobey92}. It has also been found to be an important tool for computing p-values in peak detection and thresholding problems in statistics \cite{SGA11,CS16sphere,CS16} and neuroimaging \cite{Chumbley:2009,Poline:1997,Worsley:1996b,Worsley:1996a}. The height distribution of local maxima of Gaussian random fields has been studied under Fyodorov's restriction in \cite{CS15}. For general critical points of index $i$, it follows from similar arguments in \cite{CS15} that,
\begin{equation}\label{Eq:F-ratio}
F_i(u) = \frac{\E[\mu_i(X, u)]}{\E[\mu_i(X)]},
\end{equation}
where $\mu_i(X) = \mu_i(X, -\infty)$. Therefore, $F_i(u)$ can be obtained immediately once the form of $\E[\mu_i(X, u)]$ is known. In fact, due to the ratio in \eqref{Eq:F-ratio}, the form of $F_i(u)$ is actually simpler, depending only on the single parameter $\kappa$ in the case of $\R^N$  (see Corollaries \ref{Cor:HeightDist-Euclidean-Nondegenerate} and \ref{Cor:HeightDist-Euclidean-Degenerate}) but the two parameters $\eta$ and $\kappa$ in the case of $\S^N$ (see Corollaries \ref{Cor:HeightDist-sphere-Nondegenerate} and \ref{Cor:HeightDist-sphere-Degenerate}).

This article is organized as follows. We first investigate in Section \ref{sec:GOI} the GOI matrices, especially the density of their ordered eigenvalues. In Section \ref{sec:isotropic-Euclidean}, we study the expected number and height distribution of critical points of isotropic Gaussian fields on Euclidean space, including the case of random Laplacian eigenfunctions, and the comparison with Fyodorov's restricted case. These studies are then extended in a parallel fashion to Gaussian fields on spheres in Section \ref{sec:isotropic-sphere}. We consider some interesting open problems for future work in Section \ref{sec:discussion}.

\section{Gaussian Orthogonally Invariant (GOI) Matrices}\label{sec:GOI}
In this section, we study a class of Gaussian random matrices called GOI, extending the well-known GOE matrices in the sense that it contains all random matrices whose distributions, like the GOE, are invariant under orthogonal transformations. We shall see in Sections \ref{sec:isotropic-Euclidean} and \ref{sec:isotropic-sphere} that the computation of the expected number of critical points of isotropic Gaussian fields can be transformed to the distribution of the ordered eigenvalues of such random matrices.

\subsection{Characterization of GOI matrices}
Recall that an $N\times N$ random matrix $H=(H_{ij})_{1\le i,j\le N}$ is said to have the \emph{Gaussian Orthogonal Ensemble} (GOE) distribution if it is symmetric and all entries are centered Gaussian variables such that
\begin{equation*}
\E[H_{ij}H_{kl}] = \frac{1}{2}(\delta_{ik}\delta_{jl} + \delta_{il}\delta_{jk}),
\end{equation*}
where $\delta_{ij}$ is the Kronecker delta function. It is well known that the GOE matrix $H$ is \emph{orthogonally invariant}, i.e., the distribution of $H$ is the same as that of $QHQ^T$ for any $N\times N$ orthogonal matrix $Q$. Moreover, the entries $(H_{ij}, 1\le i\le j\le N)$ are independent.

We call an $N\times N$ random matrix $M=(M_{ij})_{1\le i,j\le N}$ \emph{Gaussian Orthogonally Invariant} (GOI) with \emph{covariance parameter} $c$, denoted by $\GOI (c)$, if it is symmetric and all entries are centered Gaussian variables such that
\begin{equation}\label{eq:OI}
\E[M_{ij}M_{kl}] = \frac{1}{2}(\delta_{ik}\delta_{jl} + \delta_{il}\delta_{jk}) + c\delta_{ij}\delta_{kl}.
\end{equation}
Mallows \cite{Mallows61} showed that, up to a scaling constant, a symmetric Gaussian random matrix $M$ is orthogonally invariant if and only if it satisfies \eqref{eq:OI} for certain $c$. Notice that $\E[M_{ii} M_{jj}] = c$ for $i\ne j$. In other words, $c$ introduces a covariance between the diagonal entries of $M$ that is absent in the GOE. It is evident that $\GOI(c)$ becomes a GOE if $c=0$.

Throughout this paper, we denote by $\mathbf{1}_N$ and $I_N$ the $N\times 1$ column vector of ones and the $N\times N$ identity matrix respectively. The following result shows that the real covariance parameter $c$ in a GOI matrix cannot be too negative. For this, define a symmetric random matrix $M$ to be {\em nondegenerate} if the random vector of diagonal and upper (or lower) diagonal entries is nondegenerate.

\begin{lemma}\label{lemma:GOI-c}
Let $M$ be $\GOI(c)$ of size $N$. Then $c\ge -1/N$. In particular, $M$ is nondegenerate if and only if $c>-1/N$.
\end{lemma}
\begin{proof}\ Due to symmetry and the independence between diagonal and off-diagonal entries, we see that $M$ is nondegenerate if and only if the random vector of diagonal entries $(M_{11}, \ldots, M_{NN})$ is nondegenerate, which is equivalent to ${\rm det Cov}(M_{11}, \ldots, M_{NN}) >0$. It follows from \eqref{eq:OI} that
\begin{equation*}
{\rm Cov}(M_{11}, \ldots, M_{NN}) = I_N + c\mathbf{1}_N\mathbf{1}_N^T.
\end{equation*}
Therefore,
\begin{equation*}
{\rm det Cov}(M_{11}, \ldots, M_{NN}) = 1 + Nc,
\end{equation*}
yielding the desired result.
\end{proof}

As a characterization of GOI matrices, let $M$ be $\GOI(c)$ of size $N$. If $c\ge 0$, then $M$ can be represented as
\begin{equation}\label{Eq:c-positive}
M = H + \sqrt{c}\xi I_N,
\end{equation}
where $H$ is GOE of size $N$ and $\xi$ is a standard Gaussian variable independent of $H$ \cite{Mallows61,SMT08}. For $c\in [-1/N, 0)$, $M$ can be represented as
\[
M = H + \sqrt{-c}\xi' I_N,
\]
where $\xi'$ is a standard Gaussian variable such that $\E[H\xi']=-\sqrt{-c}I_N$.

As it will become clear later, Fyodorov's method \cite{Fyodorov04} is essentially based on the characterization \eqref{Eq:c-positive}, so its restriction on the covariance function translates to the constraint $c\ge 0$. Our characterization of the covariance function in Sections \ref{sec:isotropic-Euclidean} and \ref{sec:isotropic-sphere} below will include all valid values $c\in [-1/N, \infty)$.

\subsection{Density of the ordered eigenvalues of GOI matrices}
Recall \cite{Mehta:2004} that the density of the ordered eigenvalues $\la_1 \le \ldots \le \la_N$ of a GOE matrix $H$ is given by
\begin{equation}\label{Eq:GOE density}
f_0(\la_1, \ldots, \la_N)=\frac{1}{K_N} \exp\left\{ -\frac{1}{2}\sum_{i=1}^N \la_i^2 \right\} \prod_{1\leq i<j\leq N}|\la_i-\la_j|\mathbbm{1}_{\{\la_1\leq\ldots\leq\la_N\}},
\end{equation}
where the normalization constant $K_N$ can be computed from Selberg's integral
\begin{equation}\label{Eq:normalization constant}
K_N=2^{N/2}\prod_{i=1}^N\Gamma\left(\frac{i}{2}\right).
\end{equation}
We use the notation $\E_{\GOE}^N$ to represent the expectation under the GOE density \eqref{Eq:GOE density}, i.e., for a measurable function $g$,
\begin{equation}\label{Eq:E-GOE}
\E_{\GOE}^N [g(\la_1, \ldots, \la_N)] = \int_{\R^N} g(\la_1, \ldots, \la_N) f_0(\la_1, \ldots, \la_N) d\la_1\cdots d\la_N.
\end{equation}

Next, we shall derive the density of the ordered eigenvalues of a GOI matrix.
\begin{lemma}\label{lemma:GOI}
Let $M$ be an $N\times N$ nondegenerate $\GOI(c)$ matrix ($c>-1/N$). Then the density of the ordered eigenvalues $\la_1\le \ldots \le \la_N$ of $M$ is given by
\begin{equation}\label{Eq:GOI density}
\begin{split}
f_c(\la_1, \ldots, \la_N) &=\frac{1}{K_N \sqrt{1+Nc}} \exp\left\{ -\frac{1}{2}\sum_{i=1}^N \la_i^2 + \frac{c}{2(1+Nc)} \left(\sum_{i=1}^N \la_i\right)^2\right\}\\
&\quad \times  \prod_{1\le i<j \le N} |\la_i - \la_j| \mathbbm{1}_{\{\la_1\leq\ldots\leq\la_N\}},
\end{split}
\end{equation}
where $K_N$ is given in \eqref{Eq:normalization constant}.
\end{lemma}
\begin{proof}\
Define the operator ``vec'' that takes the diagonal and above-diagonal entries of $M$ as a new vector, i.e.
$$
{\rm vec}(M) = (M_{11}, \cdots, M_{NN}, M_{ij}, 1\le i\le j\le N)^T.
$$
By \eqref{eq:OI}, we have
\begin{equation*}
\begin{split}
\Sigma:=\E\left[{\rm vec}(M){\rm vec}(M)^T\right] = \left(
     \begin{array}{cc}
       I_N + c\mathbf{1}_N\mathbf{1}_N^T & 0 \\
       0 & \frac{1}{2}I_{N(N-1)/2} \\
     \end{array}
   \right),
\end{split}
\end{equation*}
such that
\[
{\rm det}(\Sigma)=\frac{1+Nc}{2^{N(N-1)/2}} \quad {\rm and} \quad \Sigma^{-1}= \left(
\begin{array}{cc}
I_N -\frac{c}{1+Nc}\mathbf{1}_N\mathbf{1}_N^T & 0 \\
0 & 2I_{N(N-1)/2} \\
\end{array}
\right).
\]
Therefore
\begin{equation*}
\begin{split}
{\rm vec}(M)^T \Sigma^{-1} {\rm vec}(M) =\sum_{1\le i, j\le N} M_{ij}^2 -\frac{c}{1+Nc} \left(\sum_{i=1}^N M_{ii}\right)^2= {\rm tr}(M^2) -\frac{c}{1+Nc} ({\rm tr}(M))^2.
\end{split}
\end{equation*}
Now, we obtain the following joint density of ${\rm vec}(M)$:
\begin{equation*}
\begin{split}
f({\rm vec}(M)) = \frac{2^{N(N-1)/4}}{(2\pi)^{N(N+1)/4}\sqrt{1+Nc}}\exp\left\{ -\frac{1}{2}{\rm tr}(M^2) + \frac{c}{2(1+Nc)} ({\rm tr}(M))^2\right\}.
\end{split}
\end{equation*}
It then follows from similar arguments in Mehta \cite{Mehta:2004} that the joint density of the ordered eigenvalues of $M$ is given by
\begin{equation*}
\begin{split}
f_c(\la_1, \ldots, \la_N) &= \frac{1}{K_N\sqrt{1+Nc}} \exp\left\{ -\frac{1}{2}\sum_{i=1}^N \la_i^2 + \frac{c}{2(1+Nc)} \left(\sum_{i=1}^N \la_i\right)^2\right\} \\
&\quad \times \prod_{1\le i<j \le N} |\la_i - \la_j|\mathbbm{1}_{\{\la_1\leq\ldots\leq\la_N\}},
\end{split}
\end{equation*}
where $K_N$ is given in \eqref{Eq:normalization constant}, yielding the desired result.
\end{proof}

We use the notation $\E_{\GOI(c)}^N$ to represent the expectation under the GOI density \eqref{Eq:GOI density}, i.e., for a measurable function $g$,
\begin{equation}\label{Eq:E-GOI}
\E_{\GOI(c)}^N [g(\la_1, \ldots, \la_N)] = \int_{\R^N} g(\la_1, \ldots, \la_N) f_c(\la_1, \ldots, \la_N) d\la_1\cdots d\la_N.
\end{equation}
Notice that when $c=0$, $\GOI(c)$ becomes GOE and $f_c$ in \eqref{Eq:GOI density} becomes $f_0$ in \eqref{Eq:GOE density}, making the notations consistent.

\section{Isotropic Gaussian Random Fields on Euclidean Space}\label{sec:isotropic-Euclidean}
Let $X=\{X(t), t\in \R^N\}$ be a centered, unit-variance, smooth isotropic Gaussian random field on $\R^N$. Here and in the sequel, the smoothness assumption means that the field satisfies the condition (11.3.1) in \cite{AT07}, which is slightly stronger than $C^2$ but can be implied by $C^3$. Due to isotropy, we can write the covariance function of $X$ as $\E\{X(t)X(s)\}=\rho(\|t-s\|^2)$ for an appropriate function $\rho(\cdot): [0,\infty) \rightarrow \R$, and denote
\begin{equation}\label{Eq:kappa}
\rho'=\rho'(0), \quad \rho''=\rho''(0), \quad \eta=\sqrt{-\rho'}/\sqrt{\rho''}, \quad \kappa=-\rho'/\sqrt{\rho''}.
\end{equation}

\begin{remark}\label{remark:scaling}
The parameters in \eqref{Eq:kappa} have the following useful property with respect to scaling of the parameter space. Suppose we define a field $\tilde{X}(t) = X(a t)$ for $a > 0$ with covariance function $\E\{\tilde{X}(t)\tilde{X}(s)\} = \tilde{\rho}(\|t-s\|^2) = \rho(a^2 \|t-s\|^2)$. Then the corresponding parameters are $\tilde{\eta} = \eta/a$ and $\tilde{\kappa} = \kappa$. In other words, the parameter $\eta$ scales inversely proportionally to the scaling of the parameter space, while the parameter $\kappa$ is invariant to it.
\end{remark}
\vspace{0.3cm}

Throughout this paper, we always assume $\rho'\rho''\ne 0$, which is equivalent to the nondegeneracy of the first and second derivatives of the field (see Lemma \ref{Lem:cov of isotropic Euclidean} below). Let
\begin{equation*}
\begin{split}
X_i(t)&=\frac{\partial X(t)}{\partial t_i}, \quad \nabla X(t)= (X_1(t), \ldots, X_N(t))^T,\\
X_{ij}(t)&=\frac{\partial^2 X(t)}{\partial t_it_j}, \quad \nabla^2 X(t)= (X_{ij}(t))_{1\le i, j\le N}.
\end{split}
\end{equation*}
The following result, characterizing the covariance of $(X(t), \nabla X(t), \nabla^2 X(t))$, can be derived easily by elementary calculations; see also \cite{AzaisW08,CS15}.

\begin{lemma}\label{Lem:cov of isotropic Euclidean} Let $\{X(t), t\in \R^N\}$ be a centered, unit-variance, smooth isotropic Gaussian random field. Then for each $t\in \R^N$ and $i$, $j$, $k$, $l\in\{1,\ldots, N\}$,
\begin{equation*}
\begin{split}
&\E\{X_i(t)X(t)\}=\E\{X_i(t)X_{jk}(t)\}=0, \quad \E\{X_i(t)X_j(t)\}=-\E\{X_{ij}(t)X(t)\}=-2\rho'\delta_{ij},\\
&\E\{X_{ij}(t)X_{kl}(t)\}=4\rho''(\delta_{ij}\delta_{kl} + \delta_{ik}\delta_{jl} + \delta_{il}\delta_{jk}),
\end{split}
\end{equation*}
where $\rho'$ and $\rho''$ are defined in \eqref{Eq:kappa}.
\end{lemma}
In particular, by Lemma \ref{Lem:cov of isotropic Euclidean}, we see that ${\rm Var}(X_i(t))=-2\rho'$ and ${\rm Var}(X_{ii}(t))=12\rho''$ for any $i\in\{1,\ldots, N\}$, which implies $\rho'<0$ and $\rho''>0$ and hence $\eta>0$ and $\kappa>0$.

\subsection{The Non-Boundary Case: $0 < \kappa^2 < (N+2)/N$}
We have the following results on the constrain of $\kappa^2$ and the nondegeneracy of joint distribution of the field and its first and second derivatives.
\begin{proposition}\label{Prop:nondegenerate-field}
\ Let the assumptions in Lemma \ref{Lem:cov of isotropic Euclidean} hold. Then $\kappa^2\le (N+2)/N$. In particular, the Gaussian vector $(X(t), \nabla X(t), X_{ij}(t), 1\leq i\leq j\leq N)$ is nondegenerate if and only if $\kappa^2<(N+2)/N$.
\end{proposition}
\begin{proof}\
By Lemma \ref{Lem:cov of isotropic Euclidean}, for each $t$, $\nabla X(t)$ is independent of $(X_{ij}(t), 1\leq i< j\leq N)$, and moreover, all of them are independent of $(X(t), X_{11}(t), \ldots,  X_{NN}(t))$. Therefore, the degeneracy of $(X(t), \nabla X(t), X_{ij}(t), 1\leq i\leq j\leq N)$ is equivalent to that of $(X(t), X_{11}(t), \ldots,  X_{NN}(t))$.

Let
\begin{equation*}
\begin{split}
\mathbf{D}={\rm Cov}(X_{11}(t), \ldots, X_{NN}(t)) = \left(
               \begin{array}{cccc}
                 12\rho'' & 4\rho'' & \cdots & 4\rho''  \\
                 4\rho'' & 12\rho'' & \cdots & 4\rho''\\
                 \vdots & \vdots & \vdots & \vdots  \\
                 4\rho'' &4\rho'' & \cdots & 12\rho'' \\
               \end{array}
             \right)
\end{split}
\end{equation*}
and
\begin{equation*}
\begin{split}
\mathbf{B} &= (\E[X(t)X_{11}(t)], \E[X(t)X_{22}(t)], \cdots, \E[X(t)X_{NN}(t)]  )= (2\rho', \cdots, 2\rho' ).
\end{split}
\end{equation*}
Notice that
\begin{equation*}
\begin{split}
{\rm det Cov}(X(t), X_{11}(t), \ldots, X_{NN}(t)) &= {\rm det }\left(
                                                               \begin{array}{cc}
                                                                 1 & \mathbf{B} \\
                                                                 \mathbf{B}^T & \mathbf{D} \\
                                                               \end{array}
                                                             \right) = (1-\mathbf{B}\mathbf{D}^{-1}\mathbf{B}^T){\rm det }(\mathbf{D})
\end{split}
\end{equation*}
and
\begin{equation*}
\begin{split}
\mathbf{D}^{-1}= \frac{1}{8(N+2)\rho''}\left(
                                         \begin{array}{cccc}
                                           N+1 & -1 & \cdots & -1 \\
                                           -1 & N+1 & \cdots & -1 \\
                                           \vdots & \vdots & \vdots & \vdots \\
                                           -1 & -1 & \cdots & N+1 \\
                                         \end{array}
                                       \right).
\end{split}
\end{equation*}
Therefore,
\begin{equation*}
\begin{split}
{\rm det Cov}(X(t), X_{11}(t), \ldots, X_{NN}(t))>0 \Leftrightarrow 1-\mathbf{B}\mathbf{D}^{-1}\mathbf{B}^T=1-\frac{N}{N+2}\kappa^2>0,
\end{split}
\end{equation*}
yielding the desired result.
\end{proof}

Due to Proposition \ref{Prop:nondegenerate-field}, we will make use of the following non-boundary condition, which characterizes the nondegeneracy of the joint distribution of the field and its Hessian.

\begin{itemize}
\item[({\bf A})] $\kappa^2 < (N+2)/N$.
\end{itemize}

\begin{lemma}\label{Lem:GOE for det Hessian} Let the assumptions in Lemma \ref{Lem:cov of isotropic Euclidean} hold. Let  $\widetilde{M}$ and $M$ be $\GOI(1/2)$ and $\GOI((1-\kappa^2)/2)$ matrices respectively.

(i) The distribution of $\nabla^2X(t)$ is the same as that of $\sqrt{8\rho''}\widetilde{M}$.

(ii) The distribution of $(\nabla^2X(t)|X(t)=x)$ is the same as that of $\sqrt{8\rho''}\big[M - \big(\kappa x/\sqrt{2}\big)I_N\big]$.
\end{lemma}

\begin{proof} \ Part (i) is a direct consequence of Lemma \ref{Lem:cov of isotropic Euclidean}. For part (ii), applying Lemma \ref{Lem:cov of isotropic Euclidean} and the well-known conditional formula for Gaussian variables, we see that $(\nabla^2f(t)|f(t)=x)$ can be written as $\Delta + 2\rho'xI_N$, where $\Delta=(\Delta_{ij})_{1\leq i,j\leq N}$ is a symmetric $N\times N$ matrix with centered Gaussian entries such that
\begin{equation*}
\E\{\Delta_{ij}\Delta_{kl}\}=4\rho''(\delta_{ik}\delta_{jl} + \delta_{il}\delta_{jk}) + 4(\rho''-\rho'^2)\delta_{ij}\delta_{kl}.
\end{equation*}
Therefore, $\Delta$ has the same distribution as the random matrix $\sqrt{8\rho''}M$, completing the proof.
\end{proof}

Notice that condition $({\bf A})$ implies $(1-\kappa^2)/2>-1/N$, making the Gaussian random matrix $\GOI((1-\kappa^2)/2)$ in Lemma \ref{Lem:GOE for det Hessian} nondegenerate. Since $c = (1-\kappa^2)/2$ here, the condition for nondegeneracy in Proposition \ref{Prop:nondegenerate-field} corresponds to the GOI condition $c > -1/N$ of Lemma \ref{lemma:GOI-c}. Therefore, the parameter $0 < \kappa^2 < (N+2)/N$ covers all smooth isotropic Gaussian random fields in $\R^N$ with nondegenerate Hessian conditional on the field.

\subsection{Expected Number and Height Distribution of Critical Points}
Recall $\mu_i(X, u)$, the number of critical points of index $i$ of $X$ exceeding $u$, defined in \eqref{Eq:mu-i} and that $\mu_i(X)=\mu_i(X, -\infty)$. We have the following result for computing $\E[\mu_i(X)]$ and $\E[\mu_i(X, u)]$. 
\begin{theorem}\label{Thm:critical-points-nondegenerate-euclidean} Let $\{X(t), t\in \R^N\}$ be a centered, unit-variance, smooth isotropic Gaussian random field satisfying condition $({\bf A})$. Then for $i=0, \ldots, N$,
\begin{equation}\label{Eq:critial-Euclidean}
\begin{split}
\E[\mu_i(X)] &=  \frac{2^{N/2}}{\pi^{N/2}\eta^N} \E_{\GOI(1/2)}^N \left[\prod_{j=1}^N|\la_j|\mathbbm{1}_{\{\la_i<0<\la_{i+1}\}}\right]
\end{split}
\end{equation}
and
\begin{equation*}
\begin{split}
\E[\mu_i(X, u)]  &=\frac{2^{N/2}}{\pi^{N/2}\eta^N} \int_u^\infty \phi(x) \E_{\GOI((1-\kappa^2)/2)}^N \left[\prod_{j=1}^N\big|\la_j-\kappa x/\sqrt{2}\big|\mathbbm{1}_{\{\la_i<\kappa x/\sqrt{2}<\la_{i+1}\}}\right] dx,
\end{split}
\end{equation*}
where $\E_{\GOI(c)}^N$ is defined in \eqref{Eq:E-GOI} and $\la_0$ and $\la_{N+1}$ are regarded respectively as $-\infty$ and $\infty$ for consistency.
\end{theorem}
\begin{proof} \ By the Kac-Rice metatheorem (see Theorem 11.2.1 in \cite{AT07}) and Lemmas \ref{Lem:cov of isotropic Euclidean} and \ref{Lem:GOE for det Hessian},
\begin{equation*}
\begin{split}
\E[\mu_i(X)] &= \frac{1}{(2\pi)^{N/2} (-2\rho')^{N/2}} \E[ |\text{det} (\nabla^2 X(t))|\mathbbm{1}_{\{\text{index} (\nabla^2 X(t)) = i\}}]\\
&= \left(\frac{2\rho''}{-\pi \rho'}\right)^{N/2} \E_{\GOI(1/2)}^N \left[\prod_{j=1}^N|\la_j|\mathbbm{1}_{\{\la_i<0<\la_{i+1}\}}\right]
\end{split}
\end{equation*}
and
\begin{equation*}
\begin{split}
\E[\mu_i(X, u)]  &=\frac{1}{(2\pi)^{N/2} (-2\rho')^{N/2}} \int_u^\infty \phi(x)\E[ |\text{det} (\nabla^2 X(t))|\mathbbm{1}_{\{\text{index} (\nabla^2 X(t)) = i\}} | X(t)=x]dx\\
&=\left(\frac{2\rho''}{-\pi \rho'}\right)^{N/2} \int_u^\infty \phi(x) \E_{\GOI((1-\kappa^2)/2)}^N \left[\prod_{j=1}^N\big|\la_j-\kappa x/\sqrt{2}\big|\mathbbm{1}_{\{\la_i<\kappa x/\sqrt{2}<\la_{i+1}\}}\right] dx.
\end{split}
\end{equation*}
Then the desired results follow from the definition of $\eta$.
\end{proof}

Recall $F_i(u)$, the height distribution of a critical point of index $i$ of $X$, defined in \eqref{Eq:F} or \eqref{Eq:F-ratio}.
Denote by $h$ the corresponding density, that is, $h_i(u)=-F_i'(u)$. The following result is an immediate consequence of \eqref{Eq:F-ratio} and Theorem \ref{Thm:critical-points-nondegenerate-euclidean}. 
\begin{corollary} \label{Cor:HeightDist-Euclidean-Nondegenerate}
Let the assumptions in Theorem \ref{Thm:critical-points-nondegenerate-euclidean} hold. Then for $i=0, \ldots, N$,
\begin{equation*}
\begin{split}
F_i(u)  &=\frac{\int_u^\infty \phi(x) \E_{\GOI((1-\kappa^2)/2)}^N \left[\prod_{j=1}^N\big|\la_j-\kappa x/\sqrt{2}\big|\mathbbm{1}_{\{\la_i<\kappa x/\sqrt{2}<\la_{i+1}\}}\right] dx}{\E_{\GOI(1/2)}^N \left[\prod_{j=1}^N|\la_j|\mathbbm{1}_{\{\la_i<0<\la_{i+1}\}}\right] }.
\end{split}
\end{equation*}
\end{corollary}

\begin{remark}\label{remark:scaling-critical-points}
	Note that $\E[\mu_i(X)]$ in \eqref{Eq:critial-Euclidean} depends only on $\eta$, while $F_i$ and hence $h_i$ depend only on $\kappa$. By the scaling properties of $\eta$ and $\kappa$ in Remark \ref{remark:scaling}, if we transform $t$ to $a t$ for some positive constant $a$, then the number of critical points increases proportionally to $a^N$, while the height distribution does not change.
\end{remark}
\vspace{0.3cm}

All the expectations in Theorem \ref{Thm:critical-points-nondegenerate-euclidean} and Corollary \ref{Cor:HeightDist-Euclidean-Nondegenerate} can be now solved in explicit form plugging in the GOI density \eqref{Eq:GOI density} directly. Note that it is sufficient to evaluate the expectations for $i=\lfloor N/2 \rfloor,\ldots,N$ since the rest of the indices can be obtained by symmetry:
\[
\E[\mu_{N-i}(X, u)] = \E[\mu_i(X, -u)], \qquad i=0,\ldots,N.
\]
Note that if $N$ is even, then $\E[\mu_{N/2}(X, u)] = \E[\mu_{N/2}(X, -u)]$, and thus the density $h_{N/2}(x)$ of the height of a critical point of index $N/2$ is symmetric around zero.


As an example, we show the explicit formulae for the expected number and height distribution of critical points for isotropic Gaussian fields on $\R^2$ satisfying condition $({\bf A})$.
\begin{example}\label{Example:critical-points-nondegenerate-euclidean} Let the assumptions in Theorem \ref{Thm:critical-points-nondegenerate-euclidean} hold and let $N=2$. Applying Theorem \ref{Thm:critical-points-nondegenerate-euclidean} and Corollary \ref{Cor:HeightDist-Euclidean-Nondegenerate}, together with Lemma \ref{lemma:GOI}, we obtain
\begin{equation*}
\E[\mu_2(X)] = \E[\mu_0(X)] = \frac{\E[\mu_1(X)]}{2} = \frac{1}{\sqrt{3}\pi\eta^2}
\end{equation*}
and
\begin{equation*}
\begin{split}
h_2(x) &=h_0(-x)=\sqrt{3}\kappa^2(x^2-1)\phi(x)\Phi\left(\frac{\kappa x}{\sqrt{2-\kappa^2}} \right) + \frac{\kappa x\sqrt{3(2-\kappa^2)}}{2\pi}e^{-\frac{x^2}{2-\kappa^2}} \\
&\qquad \qquad \qquad +\frac{\sqrt{6}}{\sqrt{\pi(3-\kappa^2)}}e^{-\frac{3x^2}{2(3-\kappa^2)}}\Phi\left(\frac{\kappa x}{\sqrt{(3-\kappa^2)(2-\kappa^2)}} \right),\\
h_1(x) &=\frac{\sqrt{3}}{\sqrt{2\pi(3-\kappa^2)}}e^{-\frac{3x^2}{2(3-\kappa^2)}},
\end{split}
\end{equation*}
where $h_1(x)$ is noticeably a normal distribution (see Figure \ref{fig:HeightDensity}). For each $i=0,1,2$, one has
\begin{equation}\label{Eq:mu-h}
\E[\mu_i(X, u)] = \E[\mu_i(X)]F_i(u)=\E[\mu_i(X)]\int_u^\infty h_i(x)dx.
\end{equation}
\end{example}

\begin{figure}[t]
\begin{center}
\begin{tabular}{cc}
\includegraphics[trim=30 10 30 10,clip,width=2.2in]{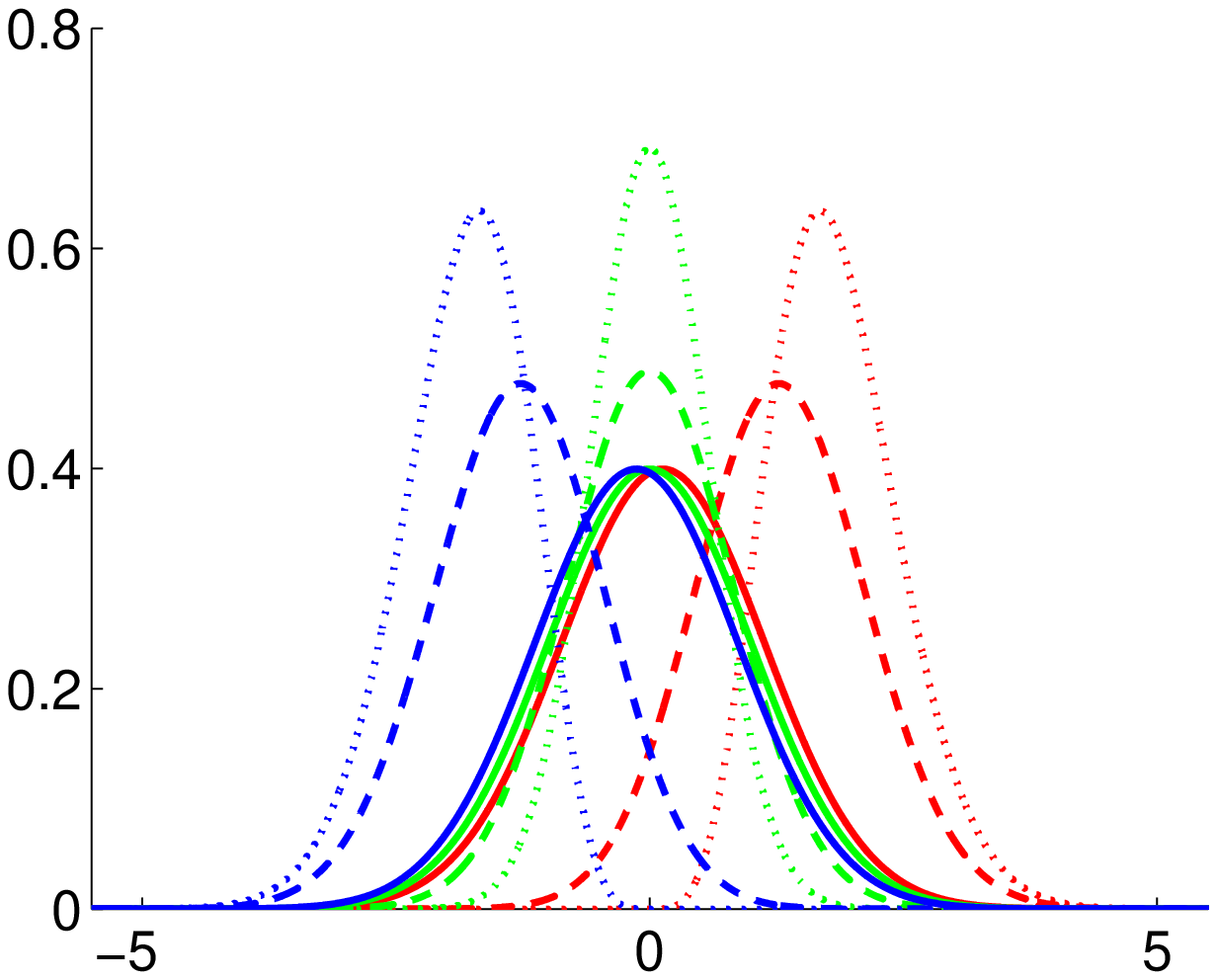} &
\includegraphics[trim=30 10 30 10,clip,width=2.2in]{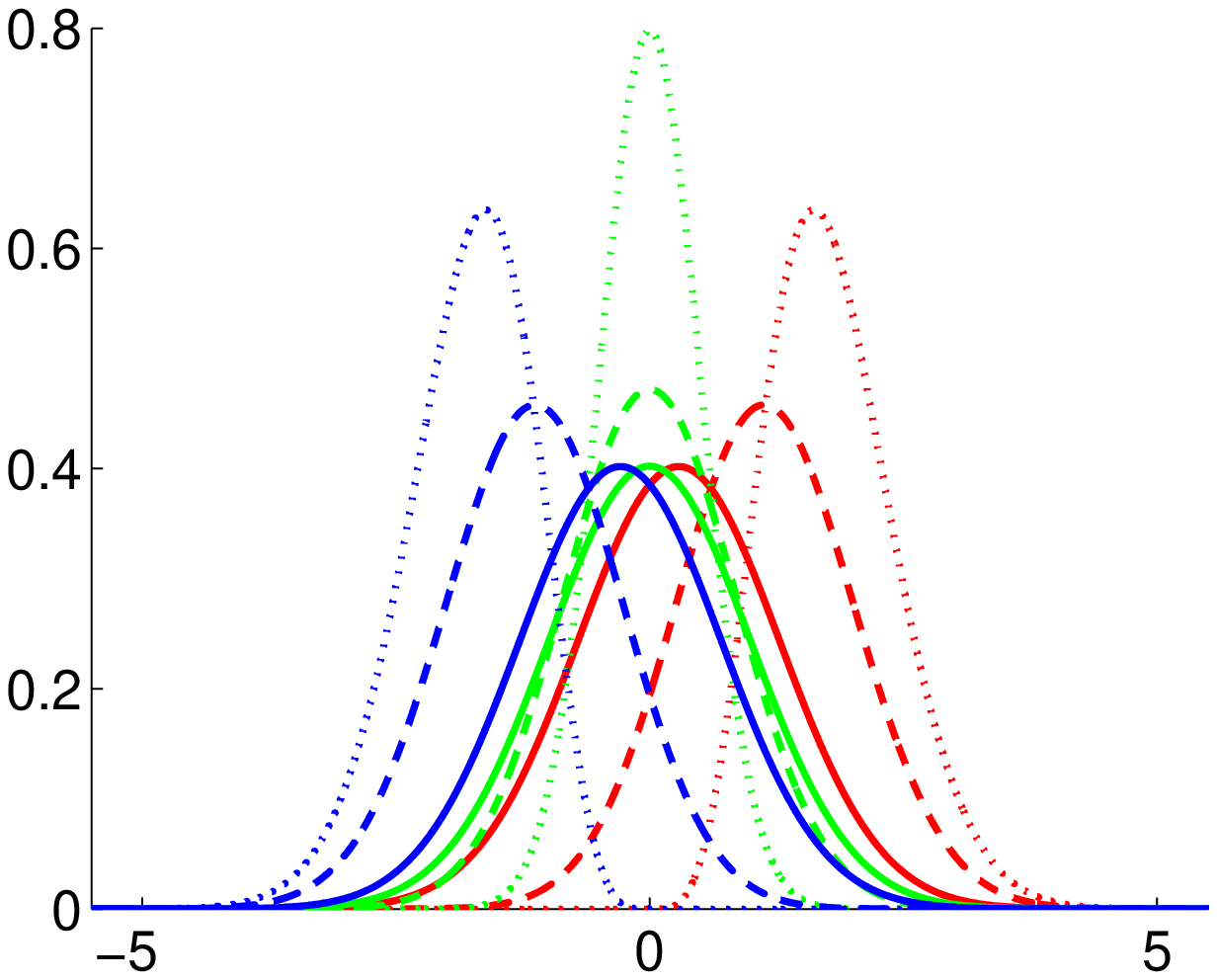}
\end{tabular}
\caption{ \label{fig:HeightDensity} \small Left panel: The height densities $h_0(x)$ (blue), $h_1(x)$ (green) and $h_2(x)$ (red) on Euclidean space $\R^2$, with parameters $\kappa=0.1$ (solid), $\kappa=1$ (dashed), and the boundary case $\kappa=\sqrt{2}$ (dotted).  Right panel: The height densities $h_0(x)$ (blue), $h_1(x)$ (green) and $h_2(x)$ (red) on the sphere $\S^2$, with parameters $\eta^2=\kappa^2=0.05$ (solid), $\eta^2=0.5$ and $\kappa^2=1$ (dashed), and the boundary case $\eta^2=1$ and $\kappa^2=3$ (dotted).}
 \end{center}
 \end{figure}

\subsection{Fyodorov's Restricted Case: $0 < \kappa^2\le 1$}\label{sec:restriction-Euclidean}

In \cite{Fyodorov04}, Fyodorov restricted his analysis to Gaussian fields in $\R^N$ whose isotropic form is valid for every dimension $N\ge 1$. It can be seen from Proposition \ref{Prop:nondegenerate-field} that $X=\{X(t), t\in \R^N\}$ is an isotropic Gaussian random field in $\R^N$ for every $N\ge 1$ if and only if
\[
\kappa^2\le \inf_{N\ge 1} \frac{N+2}{N} = 1.
\]
In that case, a $\GOI(c)$ with $c = (1-\kappa^2)/2 \ge 0$ can be written in the form \eqref{Eq:c-positive}.
Therefore, one can apply Lemma \ref{Lem:GOE for det Hessian} and the integral techniques in Fyodorov \cite{Fyodorov04} to compute $\E[\mu_i(X, u)]$.

Alternatively, under the condition $\kappa^2< 1$, we apply our general formulae in Theorem \ref{Thm:critical-points-nondegenerate-euclidean} and supporting Lemma \ref{Lem:c-positive} in the Appendix to derive immediately the following result, which is essentially the same as in Fyodorov \cite{Fyodorov04}. This shows the generality of the results in Theorem \ref{Thm:critical-points-nondegenerate-euclidean}. Notice that although the case $\kappa^2 = 1$ was not considered in \cite{Fyodorov04}, it corresponds to the pure GOE case [see \eqref{Eq:c-positive}], which is easier to handle and therefore the comparison for this case is omitted here.
\begin{proposition} Let the assumptions in Theorem \ref{Thm:critical-points-nondegenerate-euclidean} hold. Then for $i=0, \ldots, N$,
\begin{equation*}
\begin{split}
\E[\mu_i(X)] = \Gamma\left(\frac{N+1}{2}\right) \frac{2^{(N+1)/2}}{\pi^{(N+1)/2}\eta^N}\E_{\rm GOE}^{N+1}\left\{ \exp\left[-\frac{\la_{i+1}^2}{2} \right] \right\},
\end{split}
\end{equation*}
and under $\kappa^2<1$,
\begin{equation*}
\begin{split}
\E[\mu_i(X, u)]  &=\Gamma\left(\frac{N+1}{2}\right) \frac{2^{(N+1)/2}}{\pi^{(N+1)/2}\eta^N}(1-\kappa^2)^{-1/2}\\
&\quad \times  \int_u^\infty \phi(x)\E_{\rm GOE}^{N+1}\left\{ \exp\left[\frac{\la_{i+1}^2}{2} - \frac{(\la_{i+1}-\kappa x/\sqrt{2} )^2}{1-\kappa^2} \right]\right\}dx,
\end{split}
\end{equation*}
where $\E_{\rm GOE}^{N+1}$ is defined in \eqref{Eq:E-GOE}.
\end{proposition}

\subsection{The Boundary Case of Random Laplacian Eigenfunctions: $\kappa^2 = (N+2)/N$}
Here, we shall consider the case when the condition $({\bf A})$ does not hold.

\begin{proposition}\label{Prop:degenerate-field}
\ Let the assumptions in Lemma \ref{Lem:cov of isotropic Euclidean} hold. Then $(X(t), \nabla X(t), X_{ij}(t), 1\leq i\leq j\leq N)$ is degenerate if and only if $\kappa^2=(N+2)/N$, which is equivalent to the field satisfying the partial differential equation
\begin{equation}\label{Eq: degenerate}
\sum_{i=1}^N X_{ii}(t) = 2N\rho' X(t).
\end{equation}
\end{proposition}

Relation \eqref{Eq: degenerate} is also known as \emph{Helmholtz equation}. It indicates that $X(t)$ is an eigenfunction of the Laplace operator with eigenvalue $2N \rho' < 0$. This relation will be critical for deriving the formula of $\E[\mu_i(X, u)]$ in the boundary case in Theorem \ref{Thm:critical-points-degenerate-euclidean}.

\begin{proof}\
By Proposition \ref{Prop:nondegenerate-field} and its proof, we only need to consider $(X(t), X_{11}(t), \ldots,  X_{NN}(t))$ and show the equivalence between $\kappa^2=(N+2)/N$ and \eqref{Eq: degenerate}. When $\kappa^2=(N+2)/N$, that is $\rho''=\frac{N}{N+2}\rho'^2$, we have
\begin{equation*}
\begin{split}
\Sigma:={\rm Cov}(X(t), X_{11}(t), \ldots, X_{NN}(t)) &= \left(
                                                       \begin{array}{ccccc}
                                                         1 & 2\rho' & 2\rho' & \cdots & 2\rho' \\
                                                         2\rho' & \frac{12N}{N+2}\rho'^2 & \frac{4N}{N+2}\rho'^2 & \cdots & \frac{4N}{N+2}\rho'^2 \\
                                                         2\rho' & \frac{4N}{N+2}\rho'^2 & \frac{12N}{N+2}\rho'^2 & \cdots & \frac{4N}{N+2}\rho'^2 \\
                                                         \vdots & \vdots & \vdots & \vdots & \vdots \\
                                                         2\rho' & \frac{4N}{N+2}\rho'^2 & \frac{4N}{N+2}\rho'^2 & \cdots & \frac{12N}{N+2}\rho'^2 \\
                                                       \end{array}
                                                     \right).
\end{split}
\end{equation*}
It can be check that all eigenvectors of $\Sigma$ corresponding to eigenvalue 0 can be represented as $a(-2N\rho', 1, \ldots, 1)^T$, where $a\in \R$. We can diagonalize the matrix by $\Sigma=V\La V^T$, where $\La={\rm diag}(0, \la_2, \ldots, \la_N)$ (here $\la_2, \ldots, \la_N$ are nonzero eigenvalues of $\Sigma$) and $V$ is an orthogonal matrix.

Let $Z\sim \mathcal{N}(0, I_N)$. Notice that, we can write
\begin{equation*}
\begin{split}
(X(t), X_{11}(t), \ldots, X_{NN}(t))^T = \Sigma^{1/2}Z = V\La^{1/2} V^T Z = V\La^{1/2}W,
\end{split}
\end{equation*}
where $W=V^T Z\sim \mathcal{N}(0, I_N)$. This implies
\begin{equation*}
\begin{split}
V^T(X(t), X_{11}(t), \ldots, X_{NN}(t))^T = \La^{1/2}W,
\end{split}
\end{equation*}
and hence
\begin{equation*}
\begin{split}
(-2N\rho', 1, \ldots, 1)(X(t), X_{11}(t), \ldots, X_{NN}(t))^T = 0.
\end{split}
\end{equation*}
Therefore \eqref{Eq: degenerate} holds. Conversely, it follows directly from Lemma \ref{Lem:cov of isotropic Euclidean} that \eqref{Eq: degenerate} implies $\kappa^2=(N+2)/N$.
\end{proof}

Due to Proposition \ref{Prop:degenerate-field}, we shall make use of the following condition in this section.
\begin{itemize}
\item[$({\bf A'})$] $\kappa^2 = (N+2)/N$ (or equivalently $\sum_{i=1}^N X_{ii}(t) = 2N\rho' X(t)$).
\end{itemize}

We show below two examples of isotropic Gaussian fields satisfying \eqref{Eq: degenerate}.
\begin{example}
The cosine random process on $\R$ is defined by
\[
f(t) = \xi \cos(\omega t) + \xi' \sin(\omega t),
\]
where $\xi$ and $\xi'$ are independent, standard Gaussian variables and the $\omega$ is a positive constant. It is an isotropic, unit-variance, smooth Gaussian field with covariance $C(t) = \cos(\omega t)$. In particular, $f''(t)=-\omega^2 f(t)$.
\end{example}

\begin{example}
Let $r\S^{N-1}=\{t\in \R^N, t/r \in \S^{N-1}\}$, where $r>0$.
For $t \in \R^N$, define the random field
\begin{equation*}
\begin{split}
X(t) 
&= a\int_{r\S^{N-1}} \cos\langle \omega, t\rangle \,dB(\omega) + a\int_{r\S^{N-1}} \sin\langle \omega, t\rangle \,dW(\omega),
\end{split}
\end{equation*}
where $a$ is a nonzero constant, $dB(\omega)$ and $dW(\omega)$ are independent zero-mean, unit-variance, Gaussian white noise fields. Then $X(t)$ is a zero-mean, unit-variance, isotropic Gaussian random field satisfying
\[
\sum_{i=1}^N X_{ii}(t) = - r^2 X(t).
\]
The covariance of $X$ is given by
\[
C(t,s)=\E[X(t)X(s)]=a^2\int_{r\S^{N-1}} \cos\langle \omega, t-s\rangle \,d\omega, \quad t,s\in \R^N,
\]
which is isotropic.
\end{example}

\subsection{Expected Number and Height Distribution of Critical Points}

At the boundary $\kappa^2 = (N+2)/N$, by Lemma \ref{lemma:GOI-c}, the Hessian conditional on the field in Lemma \ref{Lem:GOE for det Hessian}(ii) is a degenerate symmetric random matrix. By Lemma \ref{lemma:GOI}, the density of eigenvalues of such a random matrix degenerates as well. Consequently, the technique employed to obtain $\E[\mu_i(X, u)]$ in the proof of Theorem \ref{Thm:critical-points-degenerate-euclidean} is no longer applicable. Instead, we use below a different technique by applying the Helmholtz equation \eqref{Eq: degenerate}.

\begin{theorem}\label{Thm:critical-points-degenerate-euclidean} Let $\{X(t), t\in \R^N\}$ be a centered, unit-variance, smooth isotropic Gaussian random field satisfying condition $({\bf A'})$. Then for $i=0, \ldots, N$, $\E[\mu_i(X)]$ is given by \eqref{Eq:critial-Euclidean} and
\begin{equation*}
\begin{split}
\E[\mu_i(X, u)]  &=\frac{2^{N/2}}{\pi^{N/2}\eta^N} \E_{\GOI(1/2)}^N \left[\prod_{j=1}^N|\la_j|\mathbbm{1}_{\{\la_i<0<\la_{i+1}\}}\mathbbm{1}_{\{\sum_{j=1}^N\la_j/N\le -\sqrt{(N+2)/(2N)} u\}}\right],
\end{split}
\end{equation*}
where $\E_{\GOI(c)}^N$ is defined in \eqref{Eq:E-GOI} and $\la_0$ and $\la_{N+1}$ are regarded respectively as $-\infty$ and $\infty$ for consistency.
\end{theorem}
\begin{proof} \ We only need to prove the second part since $\E[\mu_i(X)]$ follows directly from Theorem \ref{Thm:critical-points-nondegenerate-euclidean} and $({\bf A'})$. By the Kac-Rice metatheorem, condition $({\bf A'})$ and Lemmas \ref{Lem:cov of isotropic Euclidean} and \ref{Lem:GOE for det Hessian},
\begin{equation*}
\begin{split}
\E[\mu_i(X, u)]  &=\frac{1}{(2\pi)^{N/2} (-2\rho')^{N/2}} \E[ |\text{det} (\nabla^2 X(t))|\mathbbm{1}_{\{\text{index} (\nabla^2 X(t)) = i\}}\mathbbm{1}_{\{X(t)\ge u\}}]\\
&=\frac{1}{(2\pi)^{N/2} (-2\rho')^{N/2}} \E[ |\text{det} (\nabla^2 X(t))|\mathbbm{1}_{\{\text{index} (\nabla^2 X(t)) = i\}}\mathbbm{1}_{\{\sum_{i=1}^N X_{ii}(t) \le 2N\rho'u\}}]\\
&=\frac{1}{(2\pi)^{N/2} (-2\rho')^{N/2}} \E[ |\text{det} (\nabla^2 X(t))|\mathbbm{1}_{\{\text{index} (\nabla^2 X(t)) = i\}}\mathbbm{1}_{\{{\rm Tr}(\nabla^2 X(t)) \le 2N\rho'u\}}]\\
&=\left(\frac{2\rho''}{-\pi \rho'}\right)^{N/2} \E_{\GOI(1/2)}^N \left[\left(\prod_{j=1}^N|\la_j| \right) \mathbbm{1}_{\{\la_i<0<\la_{i+1}\}}\mathbbm{1}_{\{\sqrt{8\rho''}\sum_{j=1}^N\la_j\le 2N\rho'u\}}\right]\\
&=\frac{2^{N/2}}{\pi^{N/2}\eta^N} \E_{\GOI(1/2)}^N \left[\left(\prod_{j=1}^N|\la_j| \right) \mathbbm{1}_{\{\la_i<0<\la_{i+1}\}}\mathbbm{1}_{\{\sum_{j=1}^N\la_j/N\le -\sqrt{(N+2)/(2N)} u\}}\right],
\end{split}
\end{equation*}
where the last line is due to $\eta=\sqrt{-\rho'}/\sqrt{\rho''}$ and $\kappa=-\rho'/\sqrt{\rho''}=\sqrt{(N+2)/N}$ under condition $({\bf A'})$.
\end{proof}

The following result is an immediate consequence of \eqref{Eq:F-ratio} and Theorem \ref{Thm:critical-points-degenerate-euclidean}.
\begin{corollary} \label{Cor:HeightDist-Euclidean-Degenerate}
 Let the assumptions in Theorem \ref{Thm:critical-points-degenerate-euclidean} hold. Then for $i=0, \ldots, N$,
\begin{equation*}
\begin{split}
F_i(u)  &=\frac{\E_{\GOI(1/2)}^N \left[\prod_{j=1}^N|\la_j|\mathbbm{1}_{\{\la_i<0<\la_{i+1}\}}\mathbbm{1}_{\{\sum_{j=1}^N\la_j/N\le -\sqrt{(N+2)/(2N)} u\}}\right]}{\E_{\GOI(1/2)}^N \left[\prod_{j=1}^N|\la_j|\mathbbm{1}_{\{\la_i<0<\la_{i+1}\}}\right] }.
\end{split}
\end{equation*}
\end{corollary}

\begin{example}\label{Example:critical-points-degenerate-euclidean} Let the assumptions in Theorem \ref{Thm:critical-points-degenerate-euclidean} hold and let $N=2$, implying $\kappa^2=2$. Then
\begin{equation*}
\E[\mu_2(X)] = \E[\mu_0(X)] = \frac{\E[\mu_1(X)]}{2} = \frac{1}{\sqrt{3}\pi\eta^2}
\end{equation*}
and
\begin{equation*}
\begin{split}
h_2(x) &=\frac{2\sqrt{3}}{\sqrt{2\pi}}\left[(x^2-1)e^{-x^2/2} + e^{-3x^2/2}\right]\mathbbm{1}_{\{x\ge 0\}},\\
h_0(x) &=\frac{2\sqrt{3}}{\sqrt{2\pi}}\left[(x^2-1)e^{-x^2/2} + e^{-3x^2/2}\right]\mathbbm{1}_{\{x\le 0\}},\\
h_1(x) &=\frac{\sqrt{3}}{\sqrt{2\pi}}e^{-\frac{3x^2}{2}};
\end{split}
\end{equation*}
see Figure \ref{fig:HeightDensity}. The expected number of critical points $\E[\mu_i(X, u)]$ for each $i=0,1,2$ can be obtained from the densities in the same way as in Example \ref{Example:critical-points-nondegenerate-euclidean}.
\end{example}

It can be seen that the densities of height distributions of critical points in Example \ref{Example:critical-points-degenerate-euclidean} are the limits of those in Example \ref{Example:critical-points-nondegenerate-euclidean} when $\kappa^2\uparrow 2$. This is because, the expected number of critical points can be written by the Kac-Rice formula, which is continuous with respect to the parameters of the covariance of the field.

\section{Isotropic Gaussian Random Fields on Spheres}\label{sec:isotropic-sphere}
Let $\mathbb{S}^N$ denote the $N$-dimensional unit sphere and let $X= \{X(t), t\in \mathbb{S}^N\}$ be a centered, unit-variance, smooth isotropic Gaussian random field on $\mathbb{S}^N$. Due to isotropy, we may write the covariance function of $X$ as $C(\l t, s\r)$, $t,s\in \mathbb{S}^N$, where $C(\cdot): [-1,1] \rightarrow \R$ is a real function and $\l \cdot, \cdot \r$ is the inner product in $\R^{N+1}$. See \cite{CS15,ChengXiao14} for more results on the covariance function of an isotropic Gaussian field on $\mathbb{S}^N$. Define
\begin{equation}\label{Def:C' and C''}
C'=C'(1), \quad C''=C''(1), \quad \eta=\sqrt{C'}/\sqrt{C''}, \quad \kappa=C'/\sqrt{C''}.
\end{equation}
Throughout this paper, we always assume $C'C''\ne 0$, which is equivalent to the nondegeneracy of the first and second derivatives of the field (see Lemma \ref{Lem:joint distribution sphere} below).

Similarly to the Euclidean case, for an orthonormal frame $\{E_i\}_{1\le i\le N}$ on $\mathbb{S}^N$, let $X_i(t)=E_i X(t)$ and $X_{ij}(t)=E_iE_j X(t)$. The following result can be derived easily by elementary calculations; see also \cite{AuffingerPhD}.
\begin{lemma}\label{Lem:joint distribution sphere}
Let $\{X(t), t\in \S^N\}$ be a centered, unit-variance, smooth isotropic Gaussian random field. Then for each $t\in \R^N$ and $i$, $j$, $k$, $l\in\{1,\ldots, N\}$,
\begin{equation*}
\begin{split}
&\E\{X_i(t)X(t)\}=\E\{X_i(t)X_{jk}(t)\}=0, \quad \E\{X_i(t)X_j(t)\}=-\E\{X_{ij}(t)X(t)\}=C'\delta_{ij},\\
&\E\{X_{ij}(t)X_{kl}(t)\}=C''(\delta_{ik}\delta_{jl} + \delta_{il}\delta_{jk}) + (C''+C')\delta_{ij}\delta_{kl},
\end{split}
\end{equation*}
where $C'$ and $C''$ are defined in \eqref{Def:C' and C''}.
\end{lemma}
It can be seen from Lemma \ref{Lem:joint distribution sphere} that $C' = {\rm Var}[X_i(t)] > 0$ and $C'' = {\rm Var}[X_{ij}(t)] > 0$ for $i\ne j$, implying that $\eta > 0$ and $\kappa > 0$.

\subsection{The Non-Boundary Case: $0 < \kappa^2 - \eta^2 < (N+2)/N$}
Similarly to Proposition \ref{Prop:nondegenerate-field}, we have the following result.
\begin{proposition}\label{Prop:nondegenerate-field-sphere}
\ Let the assumptions in Lemma \ref{Lem:joint distribution sphere} hold. Then $\kappa^2-\eta^2\le (N+2)/N$. In particular,
the Gaussian vector $(X(t), \nabla X(t), X_{ij}(t), 1\leq i\leq j\leq N)$ is nondegenerate if and only if $\kappa^2-\eta^2<(N+2)/N$.
\end{proposition}
\begin{proof}\
The desired results follow from similar arguments in the proof of Proposition \ref{Prop:nondegenerate-field} and the observation, due to Lemma \ref{Lem:joint distribution sphere}, that
\begin{equation*}
\begin{split}
{\rm Cov}(X(t), X_{11}(t), \ldots, X_{NN}(t))  = \left(
                                                       \begin{array}{ccccc}
                                                         1 & -C' & -C' & \cdots & -C' \\
                                                         -C' & 3C''+C' & C''+C' & \cdots & C''+C' \\
                                                         -C' & C''+C' & 3C''+C' & \cdots & C''+C' \\
                                                         \vdots & \vdots & \vdots & \vdots & \vdots \\
                                                         -C' & C''+C' & C''+C' & \cdots & 3C''+C' \\
                                                       \end{array}
                                                     \right).
\end{split}
\end{equation*}
\end{proof}

In view of Proposition \ref{Prop:nondegenerate-field-sphere}, here we will make use of the following non-boundary condition.
\begin{itemize}
\item[({\bf B})] $\kappa^2-\eta^2 < (N+2)/N$.
\end{itemize}

\begin{lemma}\label{Lem:GOE for det Hessian sphere} Let the assumptions in Lemma \ref{Lem:joint distribution sphere} hold. Let  $\widetilde{M}$ and $M$ be $\GOI((1+\eta^2)/2)$ and $\GOI((1+\eta^2-\kappa^2)/2)$ matrices respectively.

(i) The distribution of $\nabla^2X(t)$ is the same as that of $\sqrt{2C''}\widetilde{M}$.

(ii) The distribution of $\nabla^2X(t)|X(t)=x$ is the same as that of $\sqrt{2C''}\big[M - \big(\kappa x/\sqrt{2}\big)I_N\big]$.
\end{lemma}

\begin{proof} \ Part (i) is a direct consequence of Lemma \ref{Lem:joint distribution sphere}. For part (ii), applying Lemma \ref{Lem:cov of isotropic Euclidean} and the well-known conditional formula for Gaussian variables, we see that $(\nabla^2X(t)|X(t)=x)$ can be written as $\Delta - C'xI_N$, where $\Delta=(\Delta_{ij})_{1\leq i,j\leq N}$ is a symmetric $N\times N$ matrix with centered Gaussian entries such that
\begin{equation*}
\E\{\Delta_{ij}\Delta_{kl}\}=C''(\delta_{ik}\delta_{jl} + \delta_{il}\delta_{jk}) + (C''+C'-C'^2)\delta_{ij}\delta_{kl}.
\end{equation*}
Therefore, $\Delta$ has the same distribution as the random matrix $\sqrt{2C''}M$, completing the proof.
\end{proof}

Note that condition $({\bf B})$ implies $(1+\eta^2-\kappa^2)/2>-1/N$, making the Gaussian random matrix $\GOI((1+\eta^2-\kappa^2)/2)$ in Lemma \ref{Lem:GOE for det Hessian sphere} nondegenerate.

\subsection{Expected Number and Height Distribution of Critical Points}
\begin{theorem}\label{Thm:critical-points-nondegenerate-sphere} Let $\{X(t), t\in \S^N\}$ be a centered, unit-variance, smooth isotropic Gaussian random field satisfying condition $({\bf B})$. Then for $i=0, \ldots, N$,
\begin{equation}\label{Eq:critical-sphere}
\begin{split}
\E[\mu_i(X)] &= \frac{1}{\pi^{N/2}\eta^N} \E_{\GOI((1+\eta^2)/2)}^N \left[\prod_{j=1}^N|\la_j|\mathbbm{1}_{\{\la_i<0<\la_{i+1}\}}\right]
\end{split}
\end{equation}
and
\begin{equation*}
\begin{split}
\E[\mu_i(X, u)]  =\frac{1}{\pi^{N/2}\eta^N} \int_u^\infty \phi(x) \E_{\GOI((1+\eta^2-\kappa^2)/2)}^N \left[\prod_{j=1}^N\big|\la_j-\kappa x/\sqrt{2}\big|\mathbbm{1}_{\{\la_i<\kappa x/\sqrt{2}<\la_{i+1}\}}\right] dx,
\end{split}
\end{equation*}
where $\phi(x)$ is the density of standard Gaussian variable, $\E_{\GOI(c)}^N$ is defined in \eqref{Eq:E-GOI}, and $\la_0$ and $\la_{N+1}$ are regarded respectively as $-\infty$ and $\infty$ for consistency.
\end{theorem}
\begin{proof} \ By the Kac-Rice metatheorem and Lemmas \ref{Lem:joint distribution sphere} and \ref{Lem:GOE for det Hessian sphere},
\begin{equation*}
\begin{split}
\E[\mu_i(X)] &= \frac{1}{(2\pi)^{N/2} (C')^{N/2}} \E[ |\text{det} (\nabla^2 X(t))|\mathbbm{1}_{\{\text{index} (\nabla^2 X(t)) = i\}}]\\
&= \left(\frac{C''}{\pi C'}\right)^{N/2} \E_{\GOI((1+\eta^2)/2)}^N \left[\left(\prod_{j=1}^N|\la_j| \right) \mathbbm{1}_{\{\la_i<0<\la_{i+1}\}}\right]
\end{split}
\end{equation*}
and
\begin{equation*}
\begin{split}
\E[\mu_i(X, u)]  &=\frac{1}{(2\pi)^{N/2} (C')^{N/2}} \int_u^\infty \phi(x)\E[ |\text{det} (\nabla^2 X(t))|\mathbbm{1}_{\{\text{index} (\nabla^2 X(t)) = i\}} | X(t)=x]dx\\
&=\left(\frac{C''}{\pi C'}\right)^{N/2} \int_u^\infty \phi(x) dx\\
 &\quad \times \E_{\GOI((1+\eta^2-\kappa^2)/2)}^N \left[\left(\prod_{j=1}^N\big|\la_j-\kappa x/\sqrt{2}\big| \right) \mathbbm{1}_{\{\la_i<\kappa x/\sqrt{2}<\la_{i+1}\}}\right].
\end{split}
\end{equation*}
Then the desired results follow from the definition of $\eta$.
\end{proof}

The following result is an immediate consequence of \eqref{Eq:F-ratio} and Theorem \ref{Thm:critical-points-nondegenerate-sphere}.
\begin{corollary}\label{Cor:HeightDist-sphere-Nondegenerate}
 Let the assumptions in Theorem \ref{Thm:critical-points-nondegenerate-sphere} hold. Then for $i=0, \ldots, N$,
\begin{equation*}
\begin{split}
F_i(u)  &=\frac{\int_u^\infty \phi(x) \E_{\GOI((1+\eta^2-\kappa^2)/2)}^N \left[\prod_{j=1}^N\big|\la_j-\kappa x/\sqrt{2}\big|\mathbbm{1}_{\{\la_i<\kappa x/\sqrt{2}<\la_{i+1}\}}\right] dx}{\E_{\GOI((1+\eta^2)/2)}^N \left[\prod_{j=1}^N|\la_j|\mathbbm{1}_{\{\la_i<0<\la_{i+1}\}}\right] }.
\end{split}
\end{equation*}
\end{corollary}

\begin{example} \label{Example:critical-points-nondegenerate-sphere}
Let the assumptions in Theorem \ref{Thm:critical-points-nondegenerate-sphere} hold and let $N=2$. Then
\begin{equation*}
\E[\mu_2(X)] = \E[\mu_0(X)] = \frac{1}{4\pi} + \frac{1}{2\pi \eta^2\sqrt{3+\eta^2}}
\end{equation*}
and
\begin{equation*}
\E[\mu_1(X)] = \frac{1}{\pi \eta^2\sqrt{3+\eta^2}}.
\end{equation*}
Moreover,
\begin{equation*}
\begin{split}
h_2(x)=h_0(-x)&=\frac{2\sqrt{3+\eta^2}}{2+\eta^2\sqrt{3+\eta^2}}\Bigg\{ \left[\eta^2+\kappa^2(x^2-1)\right]\phi(x)\Phi\left(\frac{\kappa x}{\sqrt{2+\eta^2-\kappa^2}} \right) \\
&\quad+ \frac{\kappa\sqrt{(2+\eta^2-\kappa^2)}}{2\pi}xe^{-\frac{(2+\eta^2)x^2}{2(2+\eta^2-\kappa^2)}} \\
&\quad + \frac{\sqrt{2}}{\sqrt{\pi(3+\eta^2-\kappa^2)}}e^{-\frac{(3+\eta^2)x^2}{2(3+\eta^2-\kappa^2)}} \Phi\left(\frac{\kappa x}{\sqrt{(2+\eta^2-\kappa^2)(3+\eta^2-\kappa^2)}} \right) \Bigg\}
\end{split}
\end{equation*}
and
\begin{equation*}
\begin{split}
h_1(x) =\frac{\sqrt{3+\eta^2}}{\sqrt{2\pi(3+\eta^2-\kappa^2)}}e^{-\frac{(3+\eta^2)x^2}{2(3+\eta^2-\kappa^2)}},
\end{split}
\end{equation*}
which is a normal distribution (see Figure \ref{fig:HeightDensity}). For each $i=0,1,2$, one has
\[
\E[\mu_i(X, u)] = \E[\mu_i(X)]F_i(u)=\E[\mu_i(X)]\int_u^\infty h_i(x)dx.
\]

\end{example}

\subsection{Fyodorov's Restricted Case: $\kappa^2-\eta^2\le 1$}\label{sec:restriction-sphere}

It can be seen from Proposition \ref{Prop:nondegenerate-field-sphere} that $X=\{X(t), t\in \S^N\}$ is an isotropic Gaussian random field on $\S^N$ for every $N\ge 1$ if and only if
\[
\kappa^2-\eta^2\le \inf_{N\ge 1} \frac{N+2}{N} = 1.
\]
In that case, a $\GOI(c)$ with $c = (1+\eta^2-\kappa^2)/2 \ge 0$ can be written in the form \eqref{Eq:c-positive}.
Therefore, one can apply Lemma \ref{Lem:GOE for det Hessian} and the integral techniques in Fyodorov \cite{Fyodorov04} to compute $\E[\mu_i(X, u)]$.

Alternatively, under the condition $\kappa^2-\eta^2< 1$, we apply our general formulae in Theorem \ref{Thm:critical-points-nondegenerate-sphere} and supporting Lemma \ref{Lem:c-positive} in the Appendix to derive immediately the following result, which is essentially the same as those in \cite{AuffingerPhD,Fyodorov14} (applications of Fyodorov \cite{Fyodorov04} for Gaussian fields on $\S^N$). This shows the generality of the results in Theorem \ref{Thm:critical-points-nondegenerate-euclidean}. Notice that although the case $\kappa^2-\eta^2 = 1$ was not considered in \cite{AuffingerPhD,Fyodorov14}, it corresponds to the pure GOE case [see \eqref{Eq:c-positive}], which is easier to handle and therefore the comparison for this case is omitted here.

\begin{corollary} Let the assumptions in Theorem \ref{Thm:critical-points-nondegenerate-sphere} hold. Then for $i=0, \ldots, N$, 
\begin{equation*}
\begin{split}
\E[\mu_i(X)] = \frac{\sqrt{2}\Gamma\left(\frac{N+1}{2}\right)}{\pi^{(N+1)/2}\eta^N\sqrt{1+\eta^2}} \E_{\rm GOE}^{N+1}\left\{ \exp\left[\frac{\la_{i+1}^2}{2} -\frac{\la_{i+1}^2}{1+\eta^2} \right] \right\},
\end{split}
\end{equation*}
and under $\kappa^2-\eta^2<1$,
\begin{equation*}
\begin{split}
\E[\mu_i(X, u)]  &=\frac{\sqrt{2}\Gamma\left(\frac{N+1}{2}\right)}{\pi^{(N+1)/2}\eta^N\sqrt{1+\eta^2-\kappa^2}}\\
&\quad \times  \int_u^\infty \phi(x)\E_{\rm GOE}^{N+1}\left\{ \exp\left[\frac{\la_{i+1}^2}{2} - \frac{\big(\la_{i+1}-\kappa x/\sqrt{2} \big)^2}{1+\eta^2-\kappa^2} \right]\right\}dx,
\end{split}
\end{equation*}
where $\E_{\rm GOE}^{N+1}$ is defined in \eqref{Eq:E-GOE}.
\end{corollary}

\subsection{The Boundary Case of Random Laplacian Eigenfunctions on the Sphere: $\kappa^2 - \eta^2 = (N+2)/N$}
Similarly to Proposition \ref{Prop:degenerate-field}, we have the following result. Here $X(t)$ is an eigenfunction of the Laplace operator with eigenvalue $-N C'<0$.
\begin{proposition}\label{Prop:degenerate-field-sphere}
\ Let the assumptions in Lemma \ref{Lem:joint distribution sphere} hold. Then $(X(t), \nabla X(t), X_{ij}(t), 1\leq i\leq j\leq N)$ is degenerate if and only if $\kappa^2-\eta^2=(N+2)/N$, which is equivalent to
\begin{equation}\label{Eq: degenerate-sphere}
\sum_{i=1}^N X_{ii}(t) = -NC' X(t).
\end{equation}
\end{proposition}
\begin{proof}\
The desired result follows from similar arguments in the proof of Proposition \ref{Prop:degenerate-field} and the observation that
\begin{equation*}
\begin{split}
&{\rm Cov}(X(t), X_{11}(t), \ldots, X_{NN}(t)) \\
&\quad = \left(
                                                       \begin{array}{ccccc}
                                                         1 & -C' & -C' & \cdots & -C' \\
                                                         -C' & \frac{3NC'^2+(2-2N)C'}{N+2} & \frac{NC'^2+2C'}{N+2} & \cdots & \frac{NC'^2+2C'}{N+2} \\
                                                         -C' & \frac{NC'^2+2C'}{N+2} & \frac{3NC'^2+(2-2N)C'}{N+2} & \cdots & \frac{NC'^2+2C'}{N+2} \\
                                                         \vdots & \vdots & \vdots & \vdots & \vdots \\
                                                         -C' & \frac{NC'^2+2C'}{N+2} & \frac{NC'^2+2C'}{N+2} & \cdots & \frac{3NC'^2+(2-2N)C'}{N+2} \\
                                                       \end{array}
                                                     \right)
\end{split}
\end{equation*}
when $\kappa^2-\eta^2=(N+2)/N$.
\end{proof}

We shall make use of the following condition.
\begin{itemize}
\item[$({\bf B'})$] $\kappa^2-\eta^2 = (N+2)/N$ (or equivalently $\sum_{i=1}^N X_{ii}(t) = -NC' X(t)$).
\end{itemize}

The following is an example of isotropic Gaussian fields on spheres satisfying \eqref{Eq: degenerate-sphere}.
\begin{example}
Consider an isotropic Gaussian field $X=\{X(t), t\in \S^2\}$ with covariance
\[
C(t,s)=\E[X(t)X(s)] = P_\ell(\l t, s\r), \quad t,s\in \S^2,
\]
where $P_\ell$ are Legendre polynomials and $\l \cdot, \cdot \r$ denotes the usual Euclidean inner product in $\R^3$. Then
\[
X_{11}(t) + X_{22}(t) = -\ell(\ell+1) X(t).
\]
That is, $X(t)$ is a random spherical harmonic and a Laplacian eigenfunction with eigenvalue $-\ell(\ell + 1)$ \cite{CM,CMW,Wigman:2009}.
\end{example}

\subsection{Expected Number and Height Distribution of Critical Points}
\begin{theorem}\label{Thm:critical-points-degenerate-sphere} Let $\{X(t), t\in \S^N\}$ be a centered, unit-variance, smooth isotropic Gaussian random field satisfying condition $({\bf B'})$. Then for $i=0, \ldots, N$, $\E[\mu_i(X)]$ is given in \eqref{Eq:critical-sphere} and
\begin{equation*}
\begin{split}
\E[\mu_i(X, u)]  =\frac{1}{\pi^{N/2}\eta^N} \E_{\GOI((1+\eta^2)/2)}^N \left[\prod_{j=1}^N|\la_j|\mathbbm{1}_{\{\la_i<0<\la_{i+1}\}}\mathbbm{1}_{\{\sum_{j=1}^N\la_j/N\le -\sqrt{(N+2+N\eta^2)/(2N)} u\}}\right],
\end{split}
\end{equation*}
where $\E_{\GOI(c)}^N$ is defined in \eqref{Eq:E-GOI} and $\la_0$ and $\la_{N+1}$ are regarded respectively as $-\infty$ and $\infty$ for consistency.
\end{theorem}
\begin{proof} \ We only need to prove the second part since $\E[\mu_i(X)]$ follows directly from Theorem \ref{Thm:critical-points-nondegenerate-sphere} and $({\bf B'})$. By the Kac-Rice metatheorem, condition $({\bf B'})$ and Lemmas \ref{Lem:joint distribution sphere} and \ref{Lem:GOE for det Hessian sphere},
\begin{equation*}
\begin{split}
\E&[\mu_i(X, u)]  \\
&=\frac{1}{(2\pi)^{N/2} (C')^{N/2}} \E[ |\text{det} (\nabla^2 X(t))|\mathbbm{1}_{\{\text{index} (\nabla^2 X(t)) = i\}}\mathbbm{1}_{\{X(t)\ge u\}}]\\
&=\left(\frac{C''}{\pi C'}\right)^{N/2} \E_{\GOI((1+\eta^2)/2)}^N \left[\left(\prod_{j=1}^N|\la_j| \right) \mathbbm{1}_{\{\la_i<0<\la_{i+1}\}}\mathbbm{1}_{\{\sqrt{2C''}\sum_{j=1}^N\la_j\le -NC'u\}}\right]\\
&=\frac{1}{\pi^{N/2}\eta^N} \E_{\GOI((1+\eta^2)/2)}^N \left[\left(\prod_{j=1}^N|\la_j| \right) \mathbbm{1}_{\{\la_i<0<\la_{i+1}\}}\mathbbm{1}_{\{\sum_{j=1}^N\la_j/N\le -\sqrt{(N+2+N\eta^2)/(2N)} u\}}\right],
\end{split}
\end{equation*}
where the last line is due to definitions of $\eta$ and $\kappa$ and that $\kappa=\sqrt{(N+2+N\eta^2)/N}$ under condition $({\bf B'})$.
\end{proof}

The following result is an immediate consequence of \eqref{Eq:F-ratio} and Theorem \ref{Thm:critical-points-degenerate-sphere}.
\begin{corollary}\label{Cor:HeightDist-sphere-Degenerate} Let the assumptions in Theorem \ref{Thm:critical-points-degenerate-sphere} hold. Then for $i=0, \ldots, N$,
\begin{equation*}
\begin{split}
F_i(u)  &=\frac{\E_{\GOI((1+\eta^2)/2)}^N \left[\prod_{j=1}^N|\la_j|\mathbbm{1}_{\{\la_i<0<\la_{i+1}\}}\mathbbm{1}_{\{\sum_{j=1}^N\la_j/N\le -\sqrt{(N+2+N\eta^2)/(2N)} u\}}\right]}{\E_{\GOI((1+\eta^2)/2)}^N \left[\prod_{j=1}^N|\la_j|\mathbbm{1}_{\{\la_i<0<\la_{i+1}\}}\right] }.
\end{split}
\end{equation*}
\end{corollary}

\begin{example}\label{Example:critical-points-degenerate-sphere} \ Let the assumptions in Theorem \ref{Thm:critical-points-degenerate-sphere} hold and let $N=2$, implying $\kappa^2-\eta^2=2$. Then
\begin{equation*}
\E[\mu_2(X)] = \E[\mu_0(X)] = \frac{1}{4\pi} + \frac{1}{2\pi \eta^2\sqrt{3+\eta^2}}
\end{equation*}
and
\begin{equation*}
\E[\mu_1(X)] = \frac{1}{\pi \eta^2\sqrt{3+\eta^2}}.
\end{equation*}
Moreover,
\begin{equation*}
\begin{split}
h_2(x)&=\frac{2\sqrt{3+\eta^2}}{\sqrt{2\pi}\left(2+\eta^2\sqrt{3+\eta^2}\right)}\left([(\eta^2+2)x^2-2]e^{-\frac{x^2}{2}}+ 2e^{-\frac{(3+\eta^2)x^2}{2}} \right) \mathbbm{1}_{\{x\ge 0\}},\\
h_0(x)&=\frac{2\sqrt{3+\eta^2}}{\sqrt{2\pi}\left(2+\eta^2\sqrt{3+\eta^2}\right)}\left([(\eta^2+2)x^2-2]e^{-\frac{x^2}{2}}+ 2e^{-\frac{(3+\eta^2)x^2}{2}} \right) \mathbbm{1}_{\{x\le 0\}},\\
h_1(x)&=\frac{\sqrt{3+\eta^2}}{\sqrt{2\pi}}e^{-\frac{(3+\eta^2)x^2}{2}};
\end{split}
\end{equation*}
see Figure \ref{fig:HeightDensity}. The expected number of critical points $\E[\mu_i(X, u)]$ for each $i=0,1,2$ can be obtained from the densities in the same way as in Example \ref{Example:critical-points-nondegenerate-sphere}.
\end{example}

It can be seen that the densities of height distributions of critical points in Example \ref{Example:critical-points-degenerate-sphere} are the limits of those in Example \ref{Example:critical-points-nondegenerate-sphere} when $\kappa^2-\eta^2\uparrow 2$. Again, this is due to the continuity of the expectations with respect to $\eta$ and $\kappa$.

\section{Discussion}\label{sec:discussion}

In this paper, we have used GOI matrices to model the Hessian of smooth isotropic random fields, providing a tool for computing the expected number of critical points via the Kac-Rice formula. Some potential extensions are in sight.

We see from \eqref{Eq:mu-h} that $\E[\mu_i(X,u)]$ can be computed by integrating $h_i(x)$. As an extension, if one considers the expected number of critical points of $X$ with height within certain interval $I$, then such expectation can be computed similarly to \eqref{Eq:mu-h} with the integral domain $(u,\infty)$ replaced by $I$.

There are several interesting properties of the height density $h_i(x)$ that remain to be discovered. When $N$ is even, we conjecture that $h_{N/2}(x)$ may be an exact Gaussian density. It would be interesting to study the mean, variance and mode of $h_i(x)$ and to compare $h_i(x)$ and $h_j(x)$ for $i \ne j$. Due to applications, it is very useful to investigate general explicit formulas for the expected number and height density of critical points of Gaussian fields on any dimension $N$. 

The expected number and height distribution of critical points of stationary non-isotropic Gaussian random fields remain unknown in general. We think that the problem can be solved by investigating more general Gaussian random matrices beyond GOI and making connections between them and the Hessian of the field. This would be our major future work.

\section{Appendix}
The following result shows that when $c>0$, the integral for $\GOI(c)$ in \eqref{Eq:GOI-GOE}, which will be useful for deriving $\E[\mu_i(X,u)]$, can be written in terms of an expectation of GOE of size $N+1$. This result has been applied in Sections \ref{sec:restriction-Euclidean} and \ref{sec:restriction-sphere} to show that under Fyodorov's restriction, our general formulae of $\E[\mu_i(X,u)]$ obtained in Theorems \ref{Thm:critical-points-nondegenerate-euclidean} and \ref{Thm:critical-points-nondegenerate-sphere} are exactly the same as the existing results in \cite{Fyodorov04,AuffingerPhD}.
\begin{lemma}\label{Lem:c-positive}
	Let $M$ be an $N\times N$ nondegenerate $\GOI(c)$ matrix and let $a\in \R$. If $c>0$, then for any $i_0=0, \ldots, N$,
	\begin{equation}\label{Eq:GOI-GOE}
	\begin{split}
	\E[|{\rm det}(M-aI_N)|\mathbbm{1}_{\{\text{\rm index} (M-aI_N) = i_0\}}]= \frac{\Gamma\left(\frac{N+1}{2}\right)}{\sqrt{\pi c}} \E_{\rm GOE}^{N+1}\left\{\exp\left[\frac{1}{2}\la_{i_0+1}^2 - \frac{1}{2c} (\la_{i_0+1}-a)^2\right]\right\}.
	\end{split}
	\end{equation}
\end{lemma}
\begin{proof}\
	By Lemma \ref{lemma:GOI},
	\begin{equation}\label{Eq:Int-GOI-GOE}
	\begin{split}
	&\quad \E[|{\rm det}(M-aI_N)|\mathbbm{1}_{\{\text{index} (M-aI_N) = i_0\}}]\\
	&= \int_{\R^N}\left(\prod_{i=1}^N |\la_i-a|\right)\mathbbm{1}_{\{\la_{i_0}<a<\la_{i_0+1}\}}f_c(\la_1,\ldots,\la_N)d\la_1\cdots d\la_N\\
	&= \frac{1}{2^{N/2}\sqrt{1+Nc}\prod_{i=1}^N\Gamma\left(\frac{i}{2}\right)} \int_{\R^N}d\la_1\cdots d\la_N \left(\prod_{i=1}^N |\la_i-a|\right)\mathbbm{1}_{\{\la_{i_0}<a<\la_{i_0+1}\}} \\
	&\quad \times \exp\left\{ -\frac{1}{2}\sum_{i=1}^N \la_i^2 + \frac{c}{2(1+Nc)} \left(\sum_{i=1}^N \la_i\right)^2\right\}
	\prod_{1\le i<j \le N} |\la_i - \la_j| \mathbbm{1}_{\{\la_1\leq\ldots\leq\la_N\}}.
	\end{split}
	\end{equation}
	By introducing an additional integral of standard Gaussian variable, we can write \eqref{Eq:Int-GOI-GOE} as
	\begin{equation*}
	\begin{split}
	&\frac{1}{\sqrt{2\pi}2^{N/2}\sqrt{1+Nc}\prod_{i=1}^N\Gamma\left(\frac{i}{2}\right)} \int_{\R^{N+1}}\left(\prod_{i=1}^N |\la_i-a|\right) \exp\left\{ -\frac{1}{2}\sum_{i=1}^N \la_i^2 + \frac{c}{2(1+Nc)} \left(\sum_{i=1}^N \la_i\right)^2\right\}\\
	&\times \exp\left\{-\frac{\la_*^2}{2}\right\} \prod_{1\le i<j \le N} |\la_i - \la_j| \mathbbm{1}_{\{\la_0\le \la_1\leq \ldots \le \la_{i_0}<a<\la_{i_0+1}\le \ldots\leq\la_N\le \la_{N+1}\}}d\la_1\cdots d\la_Nd\la_*,
	\end{split}
	\end{equation*}
	where $\la_0=-\infty$ and $\la_{N+1}=\infty$. Make the following change of variables:
	\begin{equation*}
	\begin{split}
	\la_i &= \tilde{\la}_i - \tilde{\la}_*, \quad \forall i = 1, \ldots, N,\\
	\la_* &= \sqrt{\frac{c}{1+Nc}}\sum_{i=1}^N \tilde{\la}_i + \frac{1}{\sqrt{c(1+Nc)}} \tilde{\la}_*.
	\end{split}
	\end{equation*}
	Notice that the Jacobian of such change of variables is $\sqrt{(1+Nc)/c}$.
	Therefore, \eqref{Eq:Int-GOI-GOE} becomes
	\begin{equation*}
	\begin{split}
	&\frac{1}{\sqrt{2\pi}2^{N/2}\sqrt{c}\prod_{i=1}^N\Gamma\left(\frac{i}{2}\right)} \int_{\R^{N+1}}\left(\prod_{i=1}^N |\tilde{\la}_i - \tilde{\la}_*-a|\right) \exp\left\{ -\frac{1}{2}\sum_{i=1}^N \tilde{\la}_i^2 - \frac{1}{2c} \tilde{\la}_*^2\right\}\\
	&\quad \times \prod_{1\le i<j \le N} |\tilde{\la}_i - \tilde{\la}_j| \mathbbm{1}_{\{\tilde{\la}_0\le \tilde{\la}_1\leq \ldots \le \tilde{\la}_{i_0}<\tilde{\la}_*+a<\tilde{\la}_{i_0+1}\le \ldots\leq\tilde{\la}_N\le \tilde{\la}_{N+1}\}}d\tilde{\la}_1\cdots d\tilde{\la}_Nd\tilde{\la}_*,
	\end{split}
	\end{equation*}
	where $\tilde{\la}_0=-\infty$ and $\tilde{\la}_{N+1}=\infty$. Make the following change of variables:
	\begin{equation*}
	\begin{split}
	\la_i &= \tilde{\la}_i, \quad \forall i = 1, \ldots, i_0,\\
	\la_{i_0+1} &= \tilde{\la}_*+a,\\
	\la_{i+1} &= \tilde{\la}_i, \quad \forall i = i_0+1, \ldots, N.
	\end{split}
	\end{equation*}
	Then \eqref{Eq:Int-GOI-GOE} becomes
	\begin{equation*}
	\begin{split}
	&\quad \frac{1}{\sqrt{2\pi}2^{N/2}\sqrt{c}\prod_{i=1}^N\Gamma\left(\frac{i}{2}\right)} \int_{\R^{N+1}}\exp\left\{\frac{1}{2}\la_{i_0+1}^2 - \frac{1}{2c} (\la_{i_0+1}-a)^2\right\}\exp\left\{ -\frac{1}{2}\sum_{i=1}^{N+1} \la_i^2 \right\}\\
	&\quad \times \prod_{1\le i<j \le N+1} |\la_i - \la_j| \mathbbm{1}_{\{\la_1\leq \ldots \leq\la_N\le \la_{N+1}\}}d\la_1\cdots d\la_Nd\la_{N+1}\\
	& = \frac{\Gamma\left(\frac{N+1}{2}\right)}{\sqrt{\pi c}} \E_{\rm GOE}^{N+1}\left\{\exp\left[\frac{1}{2}\la_{i_0+1}^2 - \frac{1}{2c} (\la_{i_0+1}-a)^2\right]\right\}.
	\end{split}
	\end{equation*}
\end{proof}

\section*{Acknowledgments}
The first author thanks Yan Fyodorov for stimulating discussions on random matrices theory.

\bibliographystyle{plain}

\begin{thebibliography}{1234}
\bibitem{Adler81}
Adler, R. J. (1981). {\it The Geometry of Random Fields}. Wiley, New York.

\bibitem{AT07}
Adler, R. J. and Taylor, J. E. (2007). {\it Random Fields and Geometry}. Springer,
New York.

\bibitem{ATW12}
Adler, R. J., Taylor, J. E. and Worsley, K. J. (2012). {\it Applications of Random
Fields and Geometry: Foundations and Case Studies.} Preprint.

\bibitem{Annibale:2003}
Annibale, A., Cavagna, A., Giardina, I. and Parisi, G. (2003). Supersymmetric complexity in the Sherrington-Kirkpatrick model. {\it Phys. Rev. E}, {\bf 68},  061103.

\bibitem{AuffingerPhD}
Auffinger, A. (2011). {\it Random Matrices, Complexity of Spin Glasses and Heavy Tailed Processes}. Ph.D. Thesis, New York University.

\bibitem{Auffinger:2013a}
Auffinger, A. and  Ben Arous, G. (2013). Complexity of random smooth functions on the high-dimensional sphere. {\it Ann. Probab.}, {\bf 41},  4214--4247.

\bibitem{Auffinger:2013b}
Auffinger, A.,  Ben Arous, G. and Cerny, J. (2013). Random matrices and complexity of spin glasses. {\it Comm. Pure Appl. Math.}, {\bf 65},  165--201.

\bibitem{AzaisW08}
Aza\"is, J.-M. and Wschebor, M. (2008). A general expression for the distribution of the maximum of a
Gaussian field and the approximation of the tail. {\it Stoch. Process. Appl.},
{\bf 118},  1190--1218.

\bibitem{AzaisW10}
Aza\"is, J.-M. and Wschebor, M. (2010). Erratum to: A general expression for the distribution of
the maximum of a Gaussian field and the approximation of the tail [Stochastic Process. Appl. 118 (7) (2008) 1190--1218]. {\it Stoch. Process. Appl.},
{\bf 120},  2100--2101.

\bibitem{AstrophysJ85}
Bardeen, J. M., Bond, J. R., Kaiser, N. and Szalay A. S. (1985). The statistics of peaks of Gaussian random fields. {\it Astrophys. J.},
{\bf 304},  15--61.

\bibitem{Belyaev67}
Belyaev, Yu. K. (1967). Bursts and shines of random fields. {\it Dokl. Akad. Nauk  SSSR}, {\bf 176}, 495--497; {\it Soviet Math. Dokl.}, {\bf 8}, 1107--1109.

\bibitem{Belyaev72}
Belyaev, Yu. K. (1972). Point processes and first passage problems. {\it Proc. Sixth Berkeley Symp. on Math. Statist. and Prob., Vol. II: Probability Theory}, University of California Press, 1--17.

\bibitem{Bray07}
Bray, A. J. and Dean, D. (2007). The statistics of critical points of Gaussian fields on large-dimensional spaces. {\it Phys. Rev. Lett.}, {\bf 98}, 150201.

\bibitem{CM}
Cammarota, V. and Marinucci, D. (2016). A quantitative central limit theorem for the Euler-Poincar\'{e} characteristic of random spherical eigenfunctions. arXiv:1603.09588.

\bibitem{CMW}
Cammarota, V., Marinucci, D. and Wigman, I. (2014). On the distribution of the critical values of random spherical harmonics. arXiv:1409.1364.

\bibitem{Cavagna:1997}
Cavagna, A., Giardina, I. and Parisi, G. (1997). Stationary points of the Thouless-Anderson-Palmer free energy. {\it Phys. Rev. B}, {\bf 57},  11 251.

\bibitem{CS16sphere}
Cheng, D., Cammarota, V., Fantaye, Y., Marinucci, D. and Schwartzman, A. (2016). Multiple testing of local maxima for detection of peaks on the (celestial) sphere. arXiv:1602.08296

\bibitem{ChengXiao14}
Cheng, D. and Xiao, Y. (2014). Excursion probability of Gaussian random fields on spheres. arXiv:1401.5498

\bibitem{CS15}
Cheng, D. and Schwartzman, A. (2015). Distribution of the height of local maxima of Gaussian random fields. {\it Extremes}, {\bf 18}, 213--240.

\bibitem{CS16}
Cheng, D. and Schwartzman, A. (2016). Multiple testing of local maxima for detection of peaks in random fields. {\it Ann. Stat.}, to appear, arXiv:1405.1400

\bibitem{Chevillard}
Chevillard, L., Rhodes, R. and Vargas, V. (2013). Gaussian multiplicative chaos for symmetric isotropic matrices. {\it J. Stat. Phys.}, {\bf 150}, 678--703.

\bibitem{Chumbley:2009}
Chumbley, J. R and Friston, K. J. (2009). False discovery rate revisited: FDR and topological inference using
  Gaussian random fields. \emph{Neuroimage}, {\bf 44}, 62--70.

\bibitem{CL67}
Cram\'er, H. and Leadbetter, M. R. (1967). {\it Stationary and Related Stochastic Processes:
Sample Function Properties and Their Applications}. Wiley, New York.

\bibitem{Crisanti:2004}
Crisanti, A., Leuzzi, L., Parisi, G. and Rizzo, T. (2004). Spin-glass complexity. {\it Phys. Rev. Lett.}, {\bf 92}, 127203.

\bibitem{Fyodorov04}
Fyodorov, Y. V. (2004). Complexity of random energy landscapes, glass transition, and absolute value of the spectral determinant of random matrices. {\it Phys. Rev. Lett.}, {\bf 92}, 240601.

\bibitem{Fyodorov15}
Fyodorov, Y. V. (2015). High-dimensional random fields and random matrix theory. {\it Markov Processes Relat.}, {\bf 21}, 483--518.

\bibitem{Fyodorov14}
Fyodorov, Y. V. and Le Doussal, P. (2014). Topology trivialization and large deviations for the minimum in the simplest random optimization. {\it J. Stat. Phys.}, {\bf 154}, 466--490.

\bibitem{Fyodorov12}
Fyodorov, Y. V. and Nadal, C. (2012). Critical behavior of the number of minima of a random landscape at the glass transition point and the Tracy-Widom distribution. {\it Phys. Rev. Lett.}, {\bf 109}, 167203.

\bibitem{Fyodorov07}
Fyodorov, Y. V. and Williams, I. (2007). Replica symmetry breaking condition exposed by random matrix calculation of landscape complexity. {\it J. Stat. Phys.}, {\bf 129}, 1081--1116.

\bibitem{Genovese:2002}
Genovese, C. R., Lazar, N. A. and Nichols, T. E. (2002). Thresholding of statistical maps in functional neuroimaging using the
  false discovery rate. \emph{Neuroimage}, {\bf 15}, 870--878.

\bibitem{Gneiting13}
Gneiting, T. (2013). Strictly and non-strictly positive definite functions on spheres. {\it Bernoulli},
{\bf 19},  1327--1349.

\bibitem{Halperin:1966}
Halperin, B. I. and Lax M. (1966). Impurity-band tails in the high-density limit. I. Minimum counting methods. {\it Phys. Rev.}, {\bf 148}, 722--740.

\bibitem{AstrophysJ04}
Larson, D. L. and Wandelt, B. D. (2004). The hot and cold spots in the Wilkinson microwave anisotropy probe data are not hot and cold enough. {\it Astrophys. J.}, {\bf 613},  85--88.

\bibitem{Lindgren72}
Lindgren, G. (1972). Local maxima of Gaussian fields. {\it Ark. Mat.}, {\bf 10},  195--218.

\bibitem{Lindgren82}
Lindgren, G. (1982). Wave characteristics distributions for Gaussian waves -- wave-length, amplitude and steepness. {\it Ocean Engng.}, {\bf 9},  411--432.

\bibitem{L-H52}
Longuet-Higgins, M. S. (1952). On the statistical distribution of the heights of sea
waves. {\it J. Marine Res.}, {\bf 11}, 245--266.

\bibitem{L-H80}
Longuet-Higgins, M. S. (1980). On the statistical distribution of the heights of sea
waves: some effects of nonlinearity and finite band width. {\it J. Geophys. Res.}, {\bf 85}, 1519--1523.

\bibitem{L-H60}
Longuet-Higgins, M. S. (1960). Reflection and refraction at a random moving surface. II. Number of specular points in a Gaussian surface. {\it J. Opt. Soc. Am.}, {\bf 50}, 845--850.

\bibitem{Mallows61}
Mallows, C. L. (1961). Latent vectors of random symmetric matrices. {\it Biometrika}, {\bf 48},  133--149.

\bibitem{MP11}
Marinucci, D. and Peccati, G. (2011). {\it Random Fields on the Sphere. Representation, Limit Theorems and Cosmological Applications}. Cambridge University Press.

\bibitem{Mehta:2004}
Mehta, M. L. (2004). {\it Random Matrices}. Volume 142, Third Edition (Pure and Applied Mathematics). Academic Press.

\bibitem{Nichols:2003}
Nichols, T. E. and Hayasaka, S. (2003). Controlling the familywise error rate in functional neuroimaging: A
comparative review. \emph{Statistical Methods in Medical Research}, {\bf 12}, 419--446, 2003.

\bibitem{Nosko69}
Nosko, V. P.  (1969). The characteristics of excursions of Gaussian homogeneous random fields above a high level. {\it Proceedings USSR-Japan Symposium on Probability (Khabarovsk,1969)}, Novosibirsk, 11--18.

\bibitem{Nosko70a}
Nosko, V. P.  (1970a). On shines of Gaussian random fields. {\it Vestnik Moskov. Univ. Ser. I Mat. Meh.}, {\bf 25}, 18--22.

\bibitem{Nosko70b}
Nosko, V. P.  (1970b). {\it The investigation of level excursions of random processes and fields}. Unpublished dissertation for the degree of Candidate of Physics and Mathematics, Lomonosov Moscow State University, Faculty of Mechanics and Mathematics.

\bibitem{Piterbarg96}
Piterbarg, V. I. (1996). Rice's method for large excursions of Gaussian random fields.
{\it Technical Report NO. 478}, Center for Stochastic Processes, Univ. North Carolina.

\bibitem{Poline:1997}
Poline, J. B., Worsley, K. J., Evans, A. C. and Friston, K. J. (1997). Combining spatial extent and peak intensity to test for activations
  in functional imaging. \emph{Neuroimage}, {\bf 5}, 83--96.

\bibitem{Rice:1945}
Rice, S. O. (1945). Mathematical analysis of random noise. {\it Bell System Tech. J.}, {\bf 24}, 46--156.

\bibitem{SchneiderWeil08}
Schneider, R. and Weil, W. (2008). {\it Stochastic and integral geometry}. Probability and Its Applications. Springer-Verlag, Berlin.

\bibitem{Schoenberg42}
Schoenberg, I. J. (1942). Positive definite functions on spheres. {\it Duke Math. J.},
{\bf 9},  96--108.

\bibitem{SGA11}
Schwartzman, A., Gavrilov, Y. and Adler, R. J. (2011). Multiple testing of local maxima for detection of peaks in 1D. {\it Ann. Statist.}, {\bf 39}, 3290--3319.

\bibitem{SMT08}
Schwartzman, A., Mascarenhas, W. and Taylor J. E. (2008). Inference for Eigenvalues and Eigenvectors of Gaussian Symmetric Matrices. {\it Ann. Statist.}, {\bf 36}, 2886--2919.

\bibitem{Sobey92}
Sobey, R. J. (1992). The distribution of zero-crossing wave heights and periods in a stationary sea state. {\it Ocean Engng.}, {\bf 19},  101--118.

\bibitem{Szego75}
Szeg\"o, G. (1975). {\it Orthogonal Polynomials}. American Mathematical Society, Providence, RI.

\bibitem{Taylor:2007}
Taylor, J. E. and Worsley, K. J. (2007). Detecting sparse signals in random fields, with an application to brain mapping. \emph{J. Am. Statist. Assoc.}, {\bf 102}, 913--928.

\bibitem{Wigman:2009}
Wigman, I. (2009). On the distribution of the nodal sets of random spherical harmonics. {\it J. Math. Phys.}, {\bf 50}, 013521.

\bibitem{Worsley:1996b}
Worsley, K. J., Marrett, S., Neelin, P. and Evans, A. C. (1996). Searching scale space for activation in PET images.
\emph{Human Brain Mapping}, {\bf 4}, 74--90.

\bibitem{Worsley:1996a}
Worsley, K. J., Marrett, S., Neelin, P., Vandal, A. C., Friston, K. J. and Evans, A. C. (1996). A unified statistical approach for determining significant signals in images of cerebral activation. \emph{Human Brain Mapping}, {\bf 4}, 58--73.

\bibitem{Worsley:2004}
Worsley, K. J., Taylor, J. E., Tomaiuolo, F. and Lerch, J. (2004). Unified univariate and multivariate random field theory. \emph{Neuroimage}, {\bf 23}, 189--195.

\end{thebibliography}

\begin{small}

\end{small}

\bigskip

\begin{quote}
\begin{small}

\textsc{Dan Cheng}\\
Department of Mathematics and Statistics \\
Texas Tech University\\
1108 Memorial Circle\\
Lubbock, TX 79409, U.S.A.\\
E-mail: \texttt{cheng.stats@gmail.com}

\vspace{.1in}
		
\textsc{Armin Schwartzman}\\
Division of Biostatistics and Bioinformatics \\
University of California San Diego\\
9500 Gilman Dr.\\
La Jolla, CA 92093, U.S.A.\\
E-mail: \texttt{armins@ucsd.edu}

	\end{small}
\end{quote}

\end{document}